\documentclass[letterpaper, 10pt]{amsart}
\usepackage{pdfsync,tabularx}
\usepackage{epstopdf}
\usepackage{microtype}
\pdfoutput=1

\usepackage{pinlabel}

\usepackage[T1]{fontenc}
\usepackage{yfonts}
\usepackage[sc]{mathpazo}

\usepackage{listings}
\usepackage{xcolor}

\newcommand{\showcomments}{yes}

\newsavebox{\commentbox}
\newenvironment{com}%
{\ifthenelse{\equal{\showcomments}{yes}}%
{\footnotemark
        \begin{lrbox}{\commentbox}
        \begin{minipage}[t]{1.25in}\raggedright\sffamily\tiny
        \footnotemark[\arabic{footnote}]}
{\begin{lrbox}{\commentbox}}}%
{\ifthenelse{\equal{\showcomments}{yes}}%
{\end{minipage}\end{lrbox}\marginpar{\usebox{\commentbox}}}
{\end{lrbox}}}

\usepackage{amssymb}
\usepackage{amsmath}
\usepackage{amsthm}
\usepackage{amsfonts}
\usepackage{mathrsfs}
\usepackage{amscd}
\usepackage{graphicx}
\usepackage{srcltx}
\usepackage{enumerate}

\long\def\comment#1\endcomment{}
\usepackage{color}

\def\<{\left\langle}
\def\>{\right\rangle}

\newcommand{\la}{\langle}
\newcommand{\ra}{\rangle}

\newtheorem{thm}{Theorem}[section]
\newtheorem{prop}[thm]{Proposition}
\newtheorem{cor}[thm]{Corollary}

\newtheorem{lem}[thm]{Lemma}
\newtheorem{cvn}[thm]{Convention}

\newtheorem{thmi}{Theorem}

\newtheorem{cori}[thmi]{Corollary}

\theoremstyle{definition} \newtheorem{defn}[thm]{Definition}

\theoremstyle{remark}
\newtheorem{rmk}[thm]{Remark}

\newtheorem{question}{Question}

\numberwithin{table}{section}

\def\square{\hfill${\vcenter{\vbox{\hrule height.4pt \hbox{\vrule
width.4pt height7pt \kern7pt \vrule width.4pt} \hrule
height.4pt}}}$}

\newcommand{\tsh}[1]{\left\{\kern-.9ex\left\{#1\right\}\kern-.9ex\right\}}

\DeclareMathOperator{\naturals}{\mathbb N}

\def\TT{{\mathcal T}}

\def\<{\langle}
\def\>{\rangle}

\def\ppp{\mathcal P}

\def\pp{{\mathcal P}}

\def\jjj{\mathcal J}

\long\def\Restate#1#2#3#4{
\medskip\par\noindent
{\bf #1 \ref{#2} #3} {\it #4}\par\medskip }

\definecolor{darkgreen}{cmyk}{1,0,1,.2}

\DeclareMathOperator{\link}{Link}

\DeclareMathOperator{\expect}{\mathbb E}

\newcommand\intersect\cap
\newcommand\infinity\infty
\newcommand\wt\widetilde
\DeclareMathOperator{\prob}{\mathbb P}

\newcommand\inject\hookrightarrow
\newcommand\union\cup
\newcommand\reals{{\mathbb R}}

\newcommand{\co}{\colon\thinspace}

\newcommand\join\Lambda
\newcommand\cross\times

\newcommand\lub\vee
\newcommand\glb\wedge

\def\coxeterthick{{\mathbb T}}

\begin{document}
\title{Thickness, relative hyperbolicity, and randomness in Coxeter
groups}

\author[J. Behrstock]{Jason Behrstock}
\address{Lehman College and The Graduate Center, CUNY, New York, New York, USA}
\email{jason.behrstock@lehman.cuny.edu}
\thanks{\flushleft {Behrstock was supported as an Alfred P. Sloan
Fellow and by the National Science Foundation under Grant Number NSF
1006219.}}

\author[M.F. Hagen]{Mark F. Hagen}
\address{U. Michigan, Ann Arbor, Michigan, USA}
\email{markfhagen@gmail.com}
\thanks{\flushleft {Hagen was supported by the National Science Foundation under Grant Number NSF 1045119.}}

\author[A. Sisto]{Alessandro Sisto}
\address{ETH, Z\"{u}rich, Switzerland}
\email{sisto@math.ethz.ch}

\address{Univ. Cath. de Louvain, Louvain-la-Neuve, Belgium}
\email{pe.caprace@uclouvain.be}

\maketitle

\centerline{
\textit{\footnotesize{
With an appendix written jointly with PIERRE-EMMANUEL CAPRACE
}}
}

\date{\today}

\begin{abstract}
For right-angled Coxeter groups $W_{\Gamma}$, we obtain a condition on
$\Gamma$ that is necessary and sufficient to ensure that $W_{\Gamma}$
is \emph{thick} and thus not relatively hyperbolic.
We show that Coxeter groups which are not thick all
admit canonical minimal relatively hyperbolic structures; further,
we show that in such a structure, the
peripheral subgroups are both parabolic (in the Coxeter group-theoretic
sense) and strongly algebraically thick.
We exhibit a
polynomial-time algorithm that decides whether a right-angled Coxeter
group is thick or relatively hyperbolic. We analyze random
graphs in the Erd\'{o}s-R\'{e}nyi model and establish the asymptotic
probability that a random right-angled Coxeter group is thick.

In the joint appendix we study Coxeter groups in full generality and
there we also obtain a dichotomy whereby any such group is either strongly
algebraically thick or admits a minimal relatively hyperbolic
structure. In this study, we also introduce a notion we
call \emph{intrinsic horosphericity} which provides a dynamical obstruction to
relative hyperbolicity which generalizes thickness.
\end{abstract}




\section*{Introduction}

The notion of relative hyperbolicity was introduced by
Gromov \cite{Gromov:hyperbolic}, then developed by
Farb \cite{Farb:RelHyp}. This notion is both sufficiently general to
include many important classes of groups including all (uniform and
non-uniform) lattices in rank-one semi-simple Lie groups, yet is sufficently restrictive
that it allows for powerful geometric, algebraic, and algorithmic
results to be proven, c.f.,
\cite{ArzhantsevaMinasyanOsin:SQ-universality, Drutu:RelHyp,
DrutuSapir:Splitting, Farb:RelHyp}. Further, relatively
hyperbolicity admits numerous geometric, topological, and dynamical
formulations which are all equivalent see e.g., \cite{Bowditch:RelHyp,
Dahmani:thesis, DrutuSapir:TreeGraded, Osin:RelHyp, Sisto:metricrelhyp, Sisto:projrelhyp,
 Yaman:RelHyp}.

Let $G$ be a finitely generated group
and $\pp$ a finite collection of proper subgroups of $G$.
The group $G$ is
\emph{hyperbolic relative to the subgroups $\pp$}, if:
collapsing the left cosets of $\pp$ to finite
diameter sets, in  any (hence all)  word metric on $G$, yields a
$\delta$--hyperbolic space; and, the collection $\pp$ satisfies the
\emph{bounded coset property} which, roughly speaking,
requires that in the $\delta$--hyperbolic metric space obtained as
above any pair of quasigeodesics with the same endpoints travels
through the collapsed cosets in approximately the same manner.
The subgroups in $\pp$ are called \emph{peripheral subgroups}. We say a
group is \emph{relatively hyperbolic} when there is some collection
of subgroups for which this holds.  A collection $\pp$ of peripheral
subgroups of the relatively hyperbolic group $G$ is \emph{minimal} if
for any other relatively hyperbolic structure $(G,\mathcal Q)$ on $G$,
each $P\in\pp$ is conjugate into some $Q\in\mathcal Q$;
relatively hyperbolic groups
do not always admit minimal structures \cite[Theorem~6.3]{BDM}.
Note
that we will follow  the convention  of requiring
the subgroups to be proper, which rules out the
trivial case of $G$ being hyperbolic relative to itself. Note also that a
group $G$ is hyperbolic relative to hyperbolic subgroups if and only
if $G$ is hyperbolic.

We will also be interested in the notion of \emph{thickness} which
was introduced by Behrstock--Dru\c{t}u--Mosher as a powerful geometric
obstruction to relative hyperbolicity
which holds in many interesting
cases, including most mapping class groups, right-angled Artin
groups, lattices in higher-rank semisimple Lie groups, and elsewhere
\cite{BDM}.
Thickness is defined inductively, at the base level,
\emph{thick of order 0}, it is characterized by linear divergence.
Roughly, a group is \emph{thick of order~n} if it is a ``network of
left cosets of
subgroups'' which are thick of lower orders, essentially this means
that the union of these cosets is the entire space
and any two points in the
space can be connected by a sequence of these cosets
which successively intersect along infinite diameter subsets;
the precise definition appears in Section~\ref{subsec:thick_background}.
Thickness has proven to be
an important invariant for obtaining upper bounds on divergence and
we shall utilize this below,
c.f.,
\cite{BehrstockCharney, BehrstockHagen:cubulated1, BehrstockDrutu:thick2,
BrockMasur:WPrelhyp, Sultan:thesis}.
In a
relatively hyperbolic group any thick subgroup must be contained
inside a peripheral
subgroup, see
\cite[Corollary~7.9]{BDM} together with \cite[Theorem~4.1]{BDM}.
This fact yields the
useful application that: any
relatively hyperbolic structure in which the peripheral subgroups
are thick is a minimal relatively hyperbolic structure, see
\cite[Theorem~1.8]{DrutuSapir:TreeGraded} and
\cite[Corollary~4.7]{BDM}.

In this paper, we study thickness and relative hyperbolicity in the setting of Coxeter groups.  One reason to do so is that Coxeter groups have numerous interesting properties which make them a standard testing ground in geometric group theory.  For example, these groups are known to act properly on CAT(0) cube complexes~\cite{NibloReeves:cubed}, which allows them to be studied using the tools of CAT(0) geometry.  In particular, this connects them to the study of thickness of cubulated groups initiated in~\cite{BehrstockHagen:cubulated1}.

We first specialize to the case of right-angled Coxeter groups, the class of which is diverse; for instance, each right-angled Artin group is a finite-index subgroup of a right-angled Artin group~\cite{DavisJanuszkiewicz}.  The right-angled Coxeter group $W_{\Gamma}$ is generated by involutions indexed by vertices of the finite simplicial graph $\Gamma$; the relations are commutation relations corresponding to edges.  Right-angled Coxeter groups admit a canonical relatively hyperbolic structure in terms of thick peripheral subgroups:

\begin{thmi}[Right-angled Coxeter groups
    are thick or relatively hyperbolic]\label{thmi:rel_hyp_graph}
Let $\mathcal T$ be the class consisting of the finite simplicial
graphs $\Lambda$ such that $W_{\Lambda}$ is strongly algebraically
thick.  Then for any finite simplicial graph $\Gamma$ either:
$\Gamma\in\mathcal T$, or there exists a collection $\mathbb J$ of
induced subgraphs of $\Gamma$ such that
$\mathbb J\subset \mathcal T$  and $W_{\Gamma}$ is hyperbolic relative to the
collection $\{W_J:J\in\mathbb J\}$ and this is relatively hyperbolic
structure is minimal.
\end{thmi}

One application of this theorem is to the quasi-isometric
classification of Coxeter groups. As thickness is a quasi-isometric
invariant, this provides a way to distinguish the thick Coxeter
groups from many other groups. A more refined classification also
follows from this result using the theorem that the quasi-isometric
image of a
group which is hyperbolic relative to thick periperhal subgroups is
also hyperbolic relative to thick periperhal subgroups each of which
is quasi-isometric to one of the peripherals in the
source, see \cite[Corollary~4.8]{BDM} and \cite{Drutu:RelHyp}. Prior
to this application of Theorem~\ref{thmi:rel_hyp_graph},
the primary source of classifying right-angled
Coxeter groups was to use classification theorems in right-angled
Artin groups (i.e., \cite{BehrstockNeumann:qigraph,
BehrstockJanuszkiewiczNeumann:highdimartin,
BestvinaKleinerSageev:RAAG1}) and then apply these
by finding commensurable right-angled
Coxeter group (for instance, by applying \cite{DavisJanuszkiewicz}).

Additionally, Theorem~\ref{thmi:rel_hyp_graph}
provides an effective classification
theorem because $\mathcal T$ can be characterized combinatorially as follows:

\begin{thmi}[Combinatorial characterization of
    thick right-angled Coxeter groups]\label{thmi:thick_char}
Let $\mathcal T$ be the class of finite simplicial graphs whose corresponding right-angled Coxeter groups are strongly algebraically thick can be characterized as follows.  It is the smallest class of graphs satisfying:
\begin{enumerate}
 \item $K_{2,2}\in\mathcal T$, where $K_{2,2}$ is the complete bipartite graph on two sets of two elements, i.e., a 4-cycle.
 \item Let $\Gamma\in\mathcal T$ and let $\Lambda\subset\Gamma$ be an
 induced subgraph which is not a clique.
 Then the graph obtained from $\Gamma$
 by coning off $\Lambda$ is in $\mathcal T$.
\item Let $\Gamma_1,\Gamma_2\in\mathcal T$ and suppose there exists a graph $\Gamma$, which is not a clique, and which arises as a subgraph of each of the $\Gamma_i$. Then the union $\Lambda$ of $\Gamma_1,\Gamma_2$ along $\Gamma$ is in $\mathcal T$, and so is any graph obtained from $\Lambda$ by adding any collection of edges joining vertices in $\Gamma_1-\Gamma$ to vertices of $\Gamma_2-\Gamma$.
\end{enumerate}
\end{thmi}

Theorems~\ref{thmi:rel_hyp_graph} and~\ref{thmi:thick_char} together
imply that any thick right-angled Coxeter group is strongly
algebraically thick.  A special case of this is that $W_{\Gamma}$ is thick of order 0
if and only if the product of two infinite right-angled Coxeter groups
(see Proposition~\ref{prop:general_join} which generalizes a result
of Dani--Thomas \cite[Theorem~4.1]{DaniThomas:divcox}).

Figures~\ref{fig:graphs_in_T}~and~\ref{fig:graphs_not_in_T} illustrate
examples of graphs in and not in $\mathcal T$.  See also
Remark~\ref{rem:bounded_squares}.  The right-angled Coxeter groups
with polynomial divergence constructed by Dani--Thomas
in~\cite{DaniThomas:divcox} are strongly algebraically thick, as can
be verified either by observing that the corresponding graphs are in
$\mathcal T$, or by combining the fact that they have
subexponential divergence with Theorem~\ref{thmi:rel_hyp_graph} and
the exponential divergence of any relatively hyperbolic group.

\begin{figure}[h]
\begin{minipage}[b]{.5\textwidth}
\centering
\includegraphics[scale=0.5]{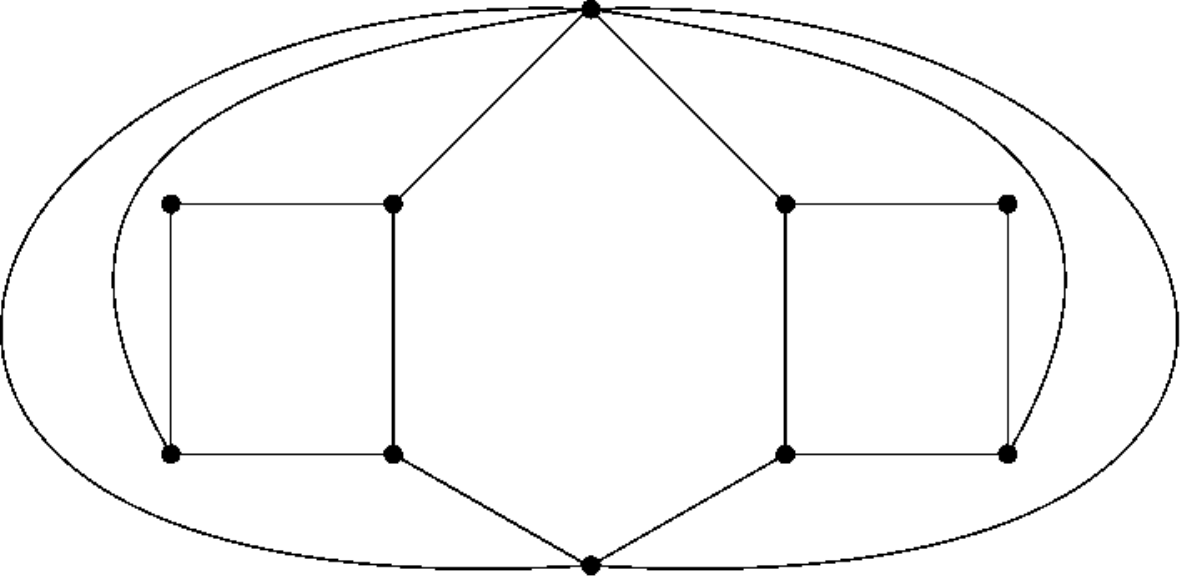}
\caption{Graph in $\mathcal T$.}\label{fig:graphs_in_T}
\end{minipage}
\hspace{-1cm}
\begin{minipage}[b]{.5\textwidth}
\centering
\includegraphics[scale=0.5]{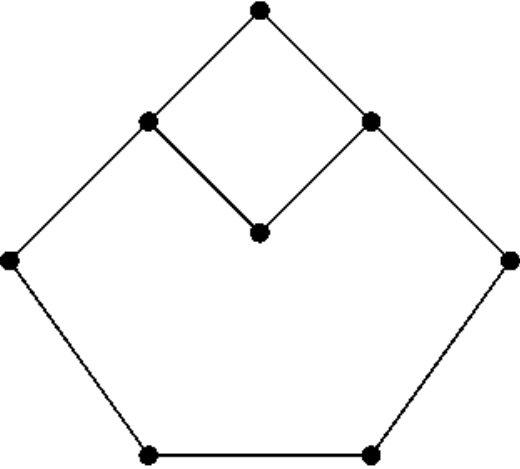}
\caption{Graph not in $\mathcal T$.}\label{fig:graphs_not_in_T}
\end{minipage}
\end{figure}

An important consequence of the above characterization of the class $\mathcal T$ is that it allows thickness/relative hyperbolicity to be detected algorithmically:

\begin{thmi}[Polynomial algorithm for relative hyperbolicity; Theorem~\ref{thm:algorithm}]\label{thmi:algorithm}
There exists a polynomial-time algorithm to decide if a given graph is
in $\mathcal T$, and hence whether a given right-angled Coxeter group
is (strongly algebraically) thick or relatively
hyperbolic.
\end{thmi}

\subsection*{Random graphs}\label{sec:intro_random}
We consider right-angled Coxeter groups on random graphs in the Erd\'{o}s--Renyi model~\cite{ErdosRenyi1}: $G(n,p(n))$ is the class of graphs on $n$ vertices with the probability measure corresponding to independently declaring each pair of vertices to be adjacent with probability $p(n)$.

An important result of Erd\'{o}s--Renyi states that a random graph is
asymptotically almost surely (a.a.s.) connected when $p(n)$ grows more
quickly that $\frac{n}{\log n}$ and is a.a.s.\ disconnected when
$p(n)=o(\frac{n}{\log n})$.  This implies that for slowly-growing
$p(n)$, when $\Gamma\in G(n,p(n))$, the right-angled Coxeter group
$W_{\Gamma}$ is a.a.s.\ a nontrivial free product, and hence
relatively hyperbolic.  In light of Theorem~\ref{thmi:rel_hyp_graph},
it is natural to wonder if there densities at which a random
right-angled Coxeter group is relatively hyperbolic but not a free
product.  The following gives a positive answer to this question; the
technical terms in this theorem will be defined in
Section~\ref{sec:random}.

\begin{thmi}[Low density, Theorem~\ref{thm:generic_relhyp}]\label{thmi:hyp_rel_free_abel}
Suppose $p(n)n\rightarrow\infty$ and $p(n)^6n^5\rightarrow0$.  Then for $\Gamma\in G(n,p(n))$, the group $W_{\Gamma}$ is a.a.s.\ hyperbolic relative to a nonempty collection of $D_{\infty}\times D_{\infty}$ subgroups, and the same holds for $W_{\Gamma'}$, where $\Gamma'\subseteq\Gamma$ is the giant component of $\Gamma$.
\end{thmi}

Intuitively, the probability of thickness should increase with the growth rate of $p(n)$, up to the point where $\Gamma$ is a.a.s.\ sufficiently dense that $W_{\Gamma}$ is either finite or virtually cyclic.  The following confirms this intuition.

\begin{thmi}[High density, Theorem~\ref{thm:high_density}]\label{thmi:high_density}
Suppose that $(1-p(n))n^2\rightarrow\alpha\in[0,\infty)$.  Then for $\Gamma\in G(n,p(n))$, the group $W_{\Gamma}$ is:
\begin{enumerate}
 \item finite with probability tending to $\beta=e^{-\alpha/2}$;
 \item virtually $\mathbb Z$ with probability tending to $\gamma=\frac{\alpha}{2}e^{-\alpha/2}$;
 \item virtually $\mathbb Z^{k},\,k\geq2$, and thus thick of order 0, with probability tending to $1-(\beta+\gamma)$.
\end{enumerate}
\end{thmi}

The following describes the situation at a natural choice of ``intermediate''~$p(n)$:

\begin{thmi}[Intermediate density]\label{thmi:constant_density}
For $\Gamma\in G(n,\frac{1}{2})$, the group $W_{\Gamma}$ is a.a.s.\ thick.
\end{thmi}

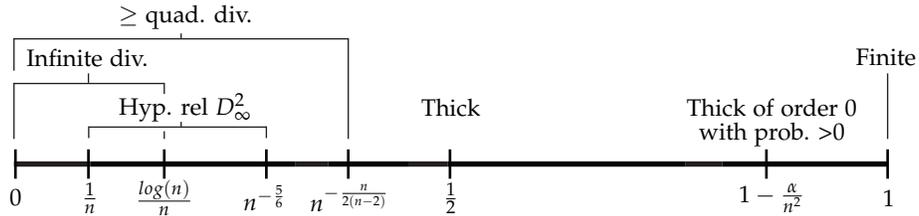
\begin{figure}[h]
     \labellist
     \small\hair 2pt
     \pinlabel {$0$} at 10 -3
     \pinlabel {$\frac{1}{n}$} at 38 -3
     \pinlabel {$\frac{log(n)}{n}$} at 66 -3
     \pinlabel {$n^{-\frac{5}{6}}$} at 104 -3
     \pinlabel {$n^{-\frac{n}{2(n-2)}}$} at 136 -3
     \pinlabel {$\frac{1}{2}$} at 174 -3
     \pinlabel {$1-\frac{\alpha}{n^{2}}$} at 295 -3
     \pinlabel {$1$} at 339 -3
     \pinlabel {Hyp. rel $D_{\infty}^{2}$} at 74 31
     \pinlabel {Thick} at 174 31
     \pinlabel {Infinite div.} at 37 50
     \pinlabel {Finite} at 338 50
     \pinlabel {Thick of order 0} at 295 31
     \pinlabel {with prob. >0} at 295 21
     \pinlabel {$\geq$ quad. div.} at 74 65
      \endlabellist
 \includegraphics[scale=1.0]{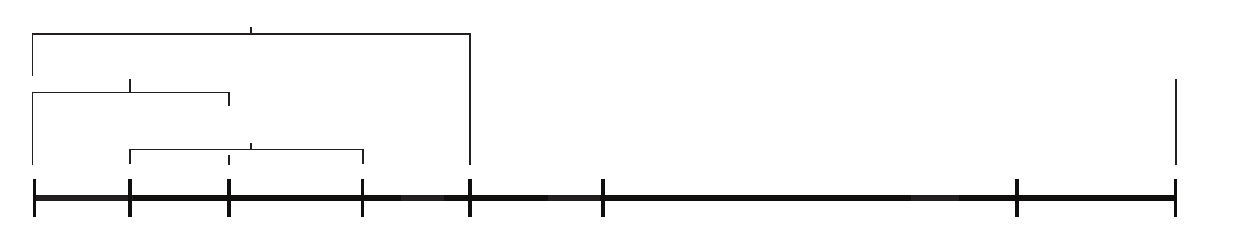}
\caption{The results of Section~\ref{sec:random} illustrated on the
same spectrum of densities as addressed conjecturally in Figure~\ref{fig:spectrum}.  The listed properties occur a.a.s.\ at
the given density, unless the specific asymptotic probability is
mentioned.}\label{fig:results3}
\end{figure}

One of our motivations for our study of random Coxeter groups was the
results of Charney and Farber on hyperbolicity of random right-angled
Coxeter groups~\cite{CharneyFarber}.  More recently, results have
been obtained about
cohomological properties of such random groups~\cite{DavisKahle:Random}.
Together with our results, this represents the beginning of a
systematic study of random Coxeter groups.

\subsection*{General Coxeter groups}\label{subsec:general_coxeter_intro}
In the Appendix, we generalize Theorem~\ref{thmi:rel_hyp_graph} and Theorem~\ref{thmi:thick_char} to all Coxeter groups; however, as shown by the example in Remark~\ref{rmk:no_generalization}, there is no characterization of strongly algebraically thick non-right-angled Coxeter groups purely in terms of the underlying graph of the free Coxeter diagram.

Theorem~\ref{thmi:rel_hyp_graph} generalizes as follows:

\begin{thmi}[Minimal relatively hyperbolic structures for Coxeter
    groups]\label{thmi:MinimalPeripheral}
Let $(W, S) $ be a Coxeter system. Then there is a
(possibly empty) collection $\jjj$ of subsets of $S$ enjoying the following properties:
\begin{enumerate}[(i)]
\item The parabolic subgroup $W_J$ is strongly algebraically thick  for every $J \in \jjj$.

\item $W$ is relatively hyperbolic with respect to $\ppp = \{W_J \; | \; J \in \jjj\}$.

\end{enumerate}

In particular $\ppp$ is a minimal relatively hyperbolic structure for $W$.
\end{thmi}

Theorem~\ref{thmi:thick_char} takes the following form for general Coxeter groups.  Note that thickness is now described using a class of labelled graphs instead of a class of graphs.

\begin{thmi}[Classification of thick Coxeter groups]\label{thmi:thickgeneral}
The class $\coxeterthick$ of Coxeter systems $(W,S)$ for which $W$ is strongly algebraically thick is the smallest class satisfying:
\begin{enumerate}
 \item $\coxeterthick$ contains the class $\coxeterthick_0$ of all irreducible \emph{affine} Coxeter systems $(W,S)$ with $S$ of cardinality~$\geq 3$, as well as all Coxeter systems of the form $(W,S_1 \cup S_2)$ with  $W_{S_1}, W_{S_2}$ irreducible non-spherical and $[W_{S_1}, W_{S_2}]=1$.
 \item Suppose that $(W, S \cup {s})$ is such that  ${s}^\perp$ is non-spherical and $(W_S, S)$ belongs to $\coxeterthick$. Then $(W, S \cup {s})$ belongs to $\coxeterthick$.
 \item Suppose that $(W,S)$ has the property that there exist $S_1,S_2\subseteq S$ with $S_1\cup S_2=S$, $(W_{S_1},S_1), (W_{S_2},S_2)\in \coxeterthick$ and  $W_{S_1\cap S_2}$ non-spherical. Then $(W,S)\in \coxeterthick$.
\end{enumerate}

\end{thmi}

We also introduce the notion, which we feel will be of independent interest, of an \emph{intrinsically horospherical} group, i.e., one for which every proper isometric action of $\Gamma$ on a proper hyperbolic geodesic metric space fixes a unique point at infinity.  Any group $G$ admits a collection of maximal intrinsically horospherical subgroups, and any relatively hyperbolic structure on $G$ has the property that every maximal intrinsically horospherical subgroup is conjugate into a peripheral subgroup.  We show that any thick group is intrinsically horospherical.  In the case of Coxeter groups, we say more:

\begin{cori}\label{cori:charThick}
Let $(W, S)$ be a Coxeter system. Then the following conditions are equivalent:
\begin{enumerate}[(I)]
\item $(W,S)$ is in $\coxeterthick$

\item $W$ is strongly algebraically thick;

\item $W$ is intrinsically horospherical;

\item $W$ is not relatively hyperbolic with respect to any family of proper subgroups.

\item $W$ is not relatively hyperbolic with respect to any family of
proper Coxeter-parabolic subgroups.

\end{enumerate}

\end{cori}

\subsection*{Outline}\label{subsec:outline}
In Section~\ref{sec:preliminaries}, we discuss background on Coxeter
groups, thickness, and divergence.
Sections~\ref{sec:thick_rel_hyp},~\ref{sec:random},
and~\ref{sec:algorithm} are devoted to right-angled Coxeter groups: in
the second section, we treat Theorems~\ref{thmi:rel_hyp_graph}
and~\ref{thmi:thick_char}.  In the third section, we study
right-angled Coxeter groups presented by random graphs, dealing in
particular with
Theorems~\ref{thmi:hyp_rel_free_abel},~\ref{thmi:high_density},
and~\ref{thmi:constant_density}.  In the fourth section, we produce an
algorithm for testing whether a given graph is in $\mathcal T$.  We
also include source code containing an implementation of a refined
version of this algorithm; this program is needed for a computation in
the proof of Theorem~\ref{thmi:constant_density}. (This source code
is available from the authors'
web pages and on the arXiv.)  In the Appendix, we
study arbitrary Coxeter groups and introduce the notion of intrinsic
horosphericity; in particular, we prove
Theorems~\ref{thmi:MinimalPeripheral} and~\ref{thmi:thickgeneral}
and Corollary~\ref{cori:charThick}.

\subsection*{Acknowledgments}\label{subsec:acknowledgements}
M.H. and A.S. thank the organizers of the conference Geometric and
Analytic Group Theory (Ventotene 2013). We thank Kaia Behrstock for
her help making Figure~\ref{fig:spectrum}.

\section{Preliminaries}\label{sec:preliminaries}
In this section, we review definitions and facts related to Coxeter groups, divergence, and thick metric spaces.  A comprehensive discussion of Coxeter groups can be found in~\cite{Davis:book}.  The notion of divergence used here is due to Gersten~\cite{GerstenDivergence}.  Our consideration of divergence in the setting of Coxeter groups was motivated largely by the discussion in~\cite{DaniThomas:divcox}, and to some extent by questions about divergence in cubulated groups (of which Coxeter groups are examples) raised in~\cite{BehrstockHagen:cubulated1}.  Thick spaces and groups were introduced in~\cite{BDM}, and we also refer to results of~\cite{BehrstockDrutu:thick2}.

\subsection{Background on Coxeter groups}\label{subsec:coxeter_background}
Throughout this paper, we confine our discussion to finitely-generated Coxeter groups.  A \emph{Coxeter group} is a group of the form $$\langle\mathcal S\mid (st)^{m_{st}}\,\co\,s,t\in\mathcal S\rangle,$$ where each $m_{ss}=1$ and for $s\neq t$, either $m_{st}\geq 2$ or there is no relation between $s,t$ of this form. Also, $m_{st}=m_{ts}$ for each $s,t\in S$. The pair $(W,\mathcal S)$ is a \emph{Coxeter system}.

The Coxeter group $W$ is \emph{reducible} if there are nonempty sets
$\mathcal S_1,\mathcal S_2\subset\mathcal S$ such that $\mathcal
S=\mathcal S_1\sqcup\mathcal S_2$, and for all $s_1\in\mathcal
S_2,s_2\in\mathcal S_2$, we have $m_{s_1s_2}=2$.  If $W$ is not
reducible, then it is \emph{irreducible}.  The Coxeter system
$(W,\mathcal S)$ is said to be \emph{(ir-)reducible} if $W$ has the
corresponding property.

To the Coxeter system $(W,\mathcal S)$, we associate a bi-linear form
$\langle-,-\rangle$ on $\mathbb R[\mathcal S]$ defined by $\langle
s,t\rangle=-\cos\left(\frac{\pi}{m_{st}}\right)$ when there is a
relation $(st^{m_{st}})$ and $\langle s,t\rangle=-1$ otherwise.  It is
well-known that this bi-linear form is positive definite if and only
if $W$ is finite, in which case the Coxeter system $(W,\mathcal S)$ is
\emph{spherical}.  Otherwise, $(W,\mathcal S)$ is non-spherical (or
\emph{aspherical}).  If the bi-linear form is positive semi-definite
and $(W,\mathcal S)$ is irreducible, then there is a short exact
sequence $\mathbb Z^n\rightarrow W\rightarrow W_0$, where
$n+1=|\mathcal S|$ and $W_0$ is a finite Coxeter group.  In this case,
the Coxeter system $(W,\mathcal S)$ is \emph{(irreducible)
affine}.

For any $J\subset\mathcal S$, the subgroup $W_J:=\langle J\rangle\subset W$ is a \emph{parabolic} subgroup.  Evidently, $W_J$ is again a Coxeter group and $(W_J,J)$ a Coxeter system.  The subset $J$ is \emph{spherical, irreducible, affine, etc.} if the Coxeter system $(W_J,J)$ has the same property.

\subsubsection{Right-angled Coxeter groups}  If each relation in the above presentation has the form $(st)^2$, then $W$ is a \emph{right-angled Coxeter group}.  In this case, let $\Gamma$ be the graph with vertex-set $\mathcal S$, and an edge joining $s,t\in\mathcal S$ if and only if $(st)^2=1$, i.e.\ if and only if the involutions $s,t$ commute.  Then $W$ decomposes as a graph product: the underlying graph is $\Gamma$, and the vertex groups are the subgroups $\langle s\rangle\cong\mathbb Z_2,\,s\in\mathcal S$.

Conversely, given a finite simplicial graph $\Gamma$ with vertex-set $\mathcal S$ and edge-set $\mathcal E$, there is a right-angled Coxeter group $$W_{\Gamma}:=\langle\mathcal S\mid s^2, (st)^2\,\co\,s,t\in\mathcal S, (s,t)\in\mathcal E\rangle.$$  For example, if $\Gamma$ is disconnected, then $W_{\Gamma}$ is isomorphic to the free product of the parabolic subgroups generated by the vertex-sets of the various components, while if $\Gamma$ decomposes as a nontrivial join, then $W_{\Gamma}$ is isomorphic to the product of the parabolic subgroups generated by the factors of the join.  For $J\subset\mathcal S$, the parabolic subgroup $W_J\leq W_{\Gamma}$ is isomorphic to the right-angled Coxeter group $W_{\Lambda}$, where $\Lambda$ is the subgraph of $\Gamma$ induced by $J$.

Finally, we remark that if $W_{\Gamma}$ is a right-angled Coxeter
group, then there exists a CAT(0) cube complex, $\widetilde X_{\Gamma}$
on which $W_\Gamma$ acts properly discontinuously and cocompactly.
This CAT(0) cube complex is the universal cover of the \emph{Davis
complex} $X_\Gamma$, which is obtained from the presentation complex of
$W_\Gamma$ by: collapsing bigons to edges, noting that each remaining
2-cell is a 2-cube, and then iteratively attaching a $k$-cubes
whenever its vertex set is contained in the $(k-1)$-skeleton, for
$k\geq 3$ (see~\cite{Davis:book} for details).
We will make use of the existence of such a CAT(0) cube complex in the proof of
Proposition~\ref{prop:general_join}.

\subsection{Background on divergence and thickness}\label{subsec:thick_background}

\newcommand{\divergence}{\mathrm{div}}
\newcommand{\Divergence}{\mathrm{Div}}

Given functions $f,g:\reals_{+}\rightarrow\reals_{+}$, we write $f\preccurlyeq g$ if for some $K\geq 1$ we have $f(s)\leq Kg(Ks+K)+Ks+K$ for all $s\in\reals_{+}$, and $f\asymp g$ if $f\preccurlyeq g$ and $g\preccurlyeq f$.

\begin{defn}[Divergence]\label{defn:divergence}
Let $(M,d)$ be a geodesic metric space, let
$\delta\in(0,1),\gamma\geq 0$, and let
$f\co\reals_{+}\rightarrow\reals_{+}$ be given by $f(r)=\delta
r-\gamma$.  Given $a,b,c\in M$ with $d(c,\{a,b\})=r>0$, let
$\divergence_f(a,b;c)=\inf\{|P|\}$, where $P$ varies over all paths
in $M$ joining $a$ to $b$ and avoiding the ball of radius $f(r)$
about $c$.  If no such path exists, $\divergence_f(a,b;c)=\infty$.
The \emph{divergence function}
$\Divergence^M_f\co\reals_{+}\rightarrow\reals_{+}$ of $M$ is then
defined by: $$\Divergence^M_f(s)=\sup\{\divergence_f(a,b;c)\co
d(a,b)\leq s\}.$$  Note that $M$ has finite divergence if and only if
$M$ has one end.
\end{defn}

Given a function $g\co\reals_{+}\rightarrow\reals_{+}$, we say that
$M$ has \emph{divergence of order at most $g$} if for some $f$ as
above, $\Divergence^M_f(s)\preccurlyeq g(s)$.  Much of the interest in
divergence comes from the fact that the divergence function of $M$ is
a quasi-isometry invariant in the sense that if $M_1$ and $M_2$ are
quasi-isometric geodesic metric spaces, and $\Divergence^{M_1}_f\asymp
g$, then $\Divergence^{M_2}_{f'}\asymp g$ for some $f'$.  In
particular, the divergence of a finitely-generated group is
well-defined up to the relation $\asymp$.
A group has linear divergence if and only if it does not have
cut-points in any asymptotic cone, such spaces are called
\emph{wide}, see \cite{Behrstock:asymptotic, DrutuMozesSapir}.

One family of metric spaces which are particularly amenable to
divergence computations are the \emph{thick} space, as
introduced in~\cite{BDM}. Thickness is a quasi-isometrically
invariant notion and this family of spaces is partitioned into
quasi-isometrically invariant subclasses by their \emph{order of
thickness}, which is a
non-negative integer.
In the present paper we work with a refinement of the
notion of thickness which is tuned for the study of
finitely generated groups:

\begin{defn}[Strongly algebraically thick~\cite{BehrstockDrutu:thick2}]\label{defn:strongly_algebraically_thick}
A finitely generated group $G$ is said to be
\emph{strongly algebraically thick
of order 0} if it is \emph{wide}.
For $n\geq 1$, the finitely generated group $G$ is
\emph{strongly algebraically thick of order at most $n$} if there
exists a finite collection $\mathcal H$ of subgroups such that:
\begin{enumerate}
 \item Each $H\in\mathcal H$ is strongly algebraically thick of order at most $n-1$.
 \item $\langle \cup_{H\in\mathcal H}H\rangle$ has finite index in $G$.
 \item There exists $C\geq 0$ such that for all $H,H'\in\mathcal H$,
 there is a sequence $H=H_1,\ldots,H_k=H'$ with each
 $H_i\in\mathcal{H}$ such that for all $i\leq k$, the intersection
 $H_i\cap H_{i+1}$ is infinite, and the $C$-neighborhood of $H_i\cap
 H_{i+1}$ (with respect to some fixed word metric on $G$) is path-connected.
 \item For all $H\in\mathcal H$, any two points in $H$ can be connected in the $C$-neighborhood of $H$ by a $(C,C)$-quasigeodesic.
\end{enumerate}
$G$ is \emph{strongly algebraically thick of order $n$} if $G$ is strongly algebraically thick of order at most $n$ but is not strongly algebraically thick of order at most $n-1$.
\end{defn}

As shown in~\cite{BehrstockDrutu:thick2}, if $G$ is strongly
algebraically thick of order $n$, then $G$, with any word metric, is a
(strongly) thick metric space.  In the present paper,
we are particularly interested in the following consequences of strong
algebraic thickness:

\begin{prop}[Upper bound on divergence; Corollary~4.17 of \cite{BehrstockDrutu:thick2}]\label{prop:divupbound}
Let $G$ be a finitely generated group that is strongly algebraically thick of order $n$.  Then the divergence function of $G$ is of order at most $s^{n+1}$.
\end{prop}

\begin{prop}[Non-relative hyperbolicity; Corollary 7.9 of \cite{BDM}]
Let $G$ be strongly algebraically thick.  Then $G$ is not hyperbolic relative to any collection of proper subgroups.
\end{prop}

Note that the above establishes that the
divergence function of thick groups is qualitatively different from
that relatively
hyperbolic groups, as the latter class has divergence functions which are at
least exponential, c.f.,  \cite[Theorem~1.3]{Sisto:metricrelhyp}.

\section{Hyperbolicity relative to thick subgroups: the right-angled case}\label{sec:thick_rel_hyp}
In this section, $\Gamma$ will denote a
finite simplicial graph and $W_{\Gamma}$ will denote
the associated right-angled Coxeter group.
We will postpone proofs of most of the results of this section to the
appendix, where we will consider them in the context of arbitrary
Coxeter groups. We focus on the right-angled case here, both
for the benefit of readers specifically interested in
the right-angled case and because these groups are cocompactly
cubulated, which allow for more refined results, such as those in
Proposition~\ref{prop:general_join} and in Section~\ref{sec:random}.

We will adopt the following:
\begin{cvn} \emph{Graph} will always mean a finite
    simplicial graph (i.e., no multi-edges or monogons). Graphs will
    often be denoted by greek letters. When we say $\Lambda$ is a
    subgraph of $\Gamma$, or write $\Lambda\subset\Gamma$, we will
    mean the \emph{full induced subgraph}, i.e., a pair of
    vertices of $\Lambda$  spans an edge in $\Lambda$ if and only if
    they span one in $\Gamma$.
\end{cvn}

%
%
%

We begin by defining the class of graphs $\mathcal T$ that we discussed
briefly in the introduction.

\begin{defn}[New graphs from old]\label{defn:new_graphs} If $\Gamma$ is a
    graph, and $\Lambda\subset\Gamma$, then we say that the graph
    $\Gamma'$ is obtained by \emph{coning off $\Lambda$} if the
    graph $\Gamma'$ can be obtained from $\Gamma$ by adding one new
    vertex along with edges between that vertex and each vertex
    of $\Lambda$. Given two graphs $\Gamma_{1}$ and $\Gamma_{2}$ with
    isomorphic subgraphs $\Gamma$, we say the \emph{union of
    $\Gamma_{1}$ and $\Gamma_{2}$ along $\Gamma$} is the graph
    obtained by taking the disjoint union of the graphs $\Gamma_{1}$
    and $\Gamma_{2}$ and identifying the corresponding $\Gamma$
    subgraphs of $\Gamma_{i}$ by
    the given isomorphism taking one of the $\Gamma$ subgraphs to the
    other. Given two graphs $\Gamma_{1}$ and $\Gamma_{2}$ with
    isomorphic subgraphs $\Gamma$, we say that a graph $\Gamma'$ is a
    \emph{generalized union of $\Gamma_{1}$ and $\Gamma_{2}$ along
    $\Gamma$} if $\Gamma'$ can be obtained from the associated union
    by adding a collection of
    edges between vertices of $\Gamma_{1}\setminus\Gamma$ and
    vertices of
    $\Gamma_{1}\setminus\Gamma$.
\end{defn}

\begin{defn}[Thick graphs]\label{defn:thick_graph}
The set of \emph{thick graphs}, $\mathcal T$, is the smallest set
of graphs satisfying the following conditions:
\begin{enumerate}
 \item $K_{2,2}\in\mathcal T$.
 \item If $\Gamma\in\mathcal T$ and $\Lambda\subset\Gamma$ is any
 induced subgraph of diameter greater than one, then the graph obtained by
 \emph{coning off $\Lambda$} is in $\mathcal T$.
 \item Let $\Gamma_1,\Gamma_2\in\mathcal T$ with both $\Gamma_{i}$
 containing an isomorphic subgraph, $\Gamma$ which is not a clique,
 then any graph which is a generalized union of the $\Gamma_{i}$ along
 $\Gamma$ is in $\mathcal T$.
\end{enumerate}
\end{defn}

When $W$ is a right-angled Coxeter group there are no irreducible
affine Coxeter systems $(W,S)$ with $S$ of cardinality~$\geq 3$.  In
particular, it is straightforward to check that a right-angled Coxeter
groups is defined by a graph in $\mathcal T$ if and only if the group
is in the class of right-angled Coxeter groups $\coxeterthick$ which
is defined
at the beginning of Section~\ref{subsec:appthickcoxeter}.  The next
result is thus a consequence of Proposition \ref{prop:thickgeneral}.

\begin{thm}\label{thm:T_is_thick}
For each $\Gamma\in\mathcal T$, the right-angled Coxeter group
$W_{\Gamma}$ is strongly algebraically thick.
\end{thm}

The main result of this section is the following which provides an
effective classification theorem with our explicit
description of~$\mathcal T$.

\begin{thm}\label{thm:thick_rel_hyp}
Let $\Gamma$ be a graph.  The right-angled Coxeter
group $W_{\Gamma}$ satisfies exactly one of the following:
\begin{itemize}
    \item it is strongly  algebraically thick and $\Gamma\in\mathcal
    T$; or,

    \item it is hyperbolic relative to a (possibly empty) minimal
    collection $\mathbb A$ of parabolic subgroups for which each
    $W_{\Lambda}\in\mathbb A$ is strongly algebraically thick and with
    each such $\Lambda\in\mathcal T$.
\end{itemize}
\end{thm}

If a group is hyperbolic relative to the empty collection of subgroups
then it is hyperbolic, hence, if $\mathbb A$ is empty then
$W_{\Gamma}$ is hyperbolic.

Theorem \ref{thm:thick_rel_hyp} can now be proven considering the
collection of all maximal subgraphs of $\Gamma$ that belong to
$\mathcal T$ and checking that conditions (RH1)--(RH3)
of~\cite[Theorem~A$'$]{Caprace:relatively_hyperbolicErr} hold.  We postpone the proof
of this to the appendix.

\begin{rmk}
An alternative way to prove Theorem~\ref{thm:thick_rel_hyp} is to define $\mathcal T$ to be the set of finite graphs whose corresponding right-angled Coxeter groups are thick.  It would then suffice to establish the following statements about induced subgraphs $J_1,J_2$ of $\Gamma$ belonging to $\mathcal T$:
\begin{enumerate}
 \item If $J_1\cap J_2$ is aspherical, then the subgraph induced by $J_1\cup J_2$ belongs to $\mathcal T$.
 \item If $v\in\Gamma-J_1$ and the link of $v$ in $J_1$ is nonempty and aspherical, then $J_1\cup\{v\}\in\mathcal T$.
 \item Joins of aspherical subgraphs belong to $\mathcal T$.
\end{enumerate}
\end{rmk}


Our explicit definition of $\mathcal T$ allows us to characterize
thick right-angled Coxeter groups, as we do now.

\begin{cor}\label{cor:thick_char}
$W_{\Gamma}$ is strongly algebraically thick if and only if $\Gamma\in\mathcal T$.
\end{cor}

\begin{proof}
If $W_{\Gamma}$ is strongly algebraically thick, then $\Gamma$ is not
relatively hyperbolic by \cite[Corollary~7.9]{BDM}.  Thus, by
Theorem~\ref{thm:thick_rel_hyp} we must have $W_{\Gamma}\in\mathcal
T$.  In the other direction: by Theorem~\ref{thm:T_is_thick}, if
$\Gamma\in\mathcal T$ then $W_{\Gamma}$ is strongly algebraically
thick.
\end{proof}

\begin{rmk}\label{rem:bounded_squares}
    From Corollary~\ref{cor:thick_char} we know that all right-angled
    Coxeter groups which are wide have corresponding graphs in $\TT$.
    As we shall see in Proposition~\ref{prop:general_join} these
    graphs all decompose as non-trivial joins, and thus in particular
    the number of squares in these graphs is linear in the
    number of vertices. In the case of right-angled Coxeter groups
    which are thick of order 1, it was proven in
    \cite{DaniThomas:divcox} that each vertex in the corresponding
    graph is contained in a square; hence in that case as well the
    number of squares is linear in the number of vertices.

    Accordingly, it is natural to expect that a graph in $\mathcal T$ contains ``many''
squares relative to the number of vertices it contains.  However, this
is not the case in general.  Indeed, for all sufficiently large
$N\in\mathbb N$ the set of graphs in $\mathcal T$ containing at most $N$ squares
is infinite.  We call a graph $\Gamma$ a \emph{filled pentagon} if
$\Gamma\in\mathcal T$ and contains vertices $v_1,\ldots,v_5$ such that
$d(v_i,v_{i+1})\geq 3$ for each $i$.  If $\Gamma$ is a filled
pentagon, then the graph obtained by joining $v_i$ and $v_{i+1}$ by a
path of length 2 is also a filled pentagon, while having the same number of squares
as $\Gamma$ and strictly more vertices.  Any element of $\mathcal T$
of diameter at least $6$ is a filled pentagon, since a path of length
6 contains a filled pentagon (as shown in
Figure~\ref{fig:path}). The claim now follows for some $N$, since
$\mathcal T$ contains graphs of arbitrarily
large diameter as we shall now show. Any graph $\Gamma\in\TT$ of
diameter at least three contains an induced path of length 2. Then
by taking the union of two copies of $\Gamma$ along this path is
still thick, by Theorem~\ref{thm:T_is_thick}, and has diameter larger
than $\Gamma$. Hence, existence of graphs in $\TT$ of arbitrarily
large diameter follows from induction and any example with diameter
at least three, e.g., as given in Figure \ref{fig:graphs_in_T}.

\end{rmk}
\begin{figure}[h]
\includegraphics[width=0.45\textwidth]{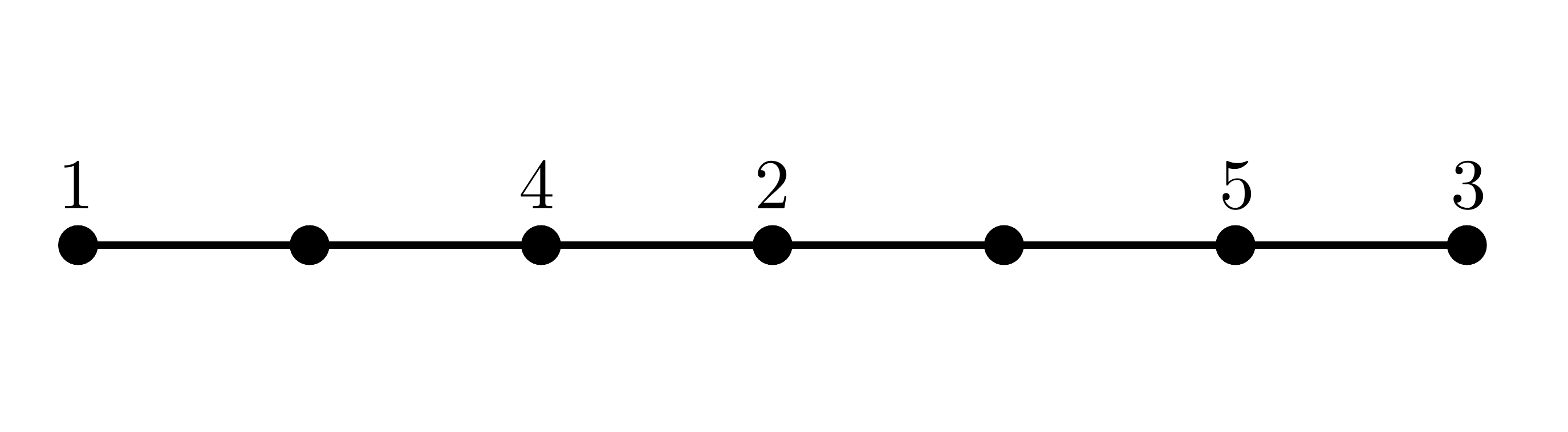}\\
\caption{A length 6 geodesic in $\Gamma$ shows that $\Gamma$ is a
filled pentagon.}\label{fig:path}
\end{figure}

\begin{rmk}[Theorem~\ref{thm:T_is_thick} does not hold for general Coxeter groups]\label{rmk:no_generalization}
Given a (not necessarily right-angled) Coxeter system $(W,\mathcal S)$, there is a naturally associated labelled graph $\Gamma$, the \emph{free Coxeter diagram}, with vertex-set $\mathcal S$ and an edge labelled $n\geq 2$ joining vertices $s,t$ that satisfy a relation $(st)^n=1$.  Note that since $m_{ss}=1$ for all $s\in\mathcal S$, this graph is simplicial.  Furthermore, if $(W,\mathcal S)$ is right-angled, then all labels are $2$ and $\Gamma$ is the graph considered above.

If the Coxeter group $W$ is not right-angled,
thickness of $W$ can not be characterized by a purely
graph-theoretic property of the
free Coxeter diagram.  Indeed, there exists a hyperbolic Coxeter group
$W$ whose free Coxeter diagram is a 4-cycle: consider the Coxeter
system determined by the presentation $$W=\langle s,t,u,v\mid
s^2,t^2,u^2,v^2, (st)^n,(su)^2,(uv)^2,(tv)^2\rangle,$$ with $n\geq 3$.
The labelled graph $\Gamma$ is a 4-cycle, with the edge joining $s,t$
labelled $n\geq 3$ and all other edges labelled 2.  However, the group
$W$ is a Fuchsian group, being generated by reflections in the sides
of a 4-gon in $\mathbb H^2$ with angles
$\frac{\pi}{2},\frac{\pi}{2},\frac{\pi}{2},\frac{\pi}{n}$.  Being
hyperbolic, $W$ cannot be thick.
\end{rmk}

Combining the upper bound on divergence of strongly thick spaces given
in~\cite[Corollary~4.17]{BehrstockDrutu:thick2},
the fact that relatively
hyperbolic groups have exponential divergence (see, e.g.,
\cite[Theorem~1.3]{Sisto:metricrelhyp}), and
Theorem~\ref{thm:thick_rel_hyp}, we obtain:

\begin{cor}\label{cor:divergence}
Let $\Gamma$ be a connected graph.  Then the
divergence function of $W_{\Gamma}$ is either 
exponential or
bounded above by a polynomial.
\end{cor}

\subsection{Characterizing thickness of order $0$}

As it turns out, the class $\mathcal T_0$ of graphs $\Gamma$ for which
$W_\Gamma$ is wide admits a simple description as we shall see below.
The triangle-free case of this results was previously established
using different techniques in
\cite[Theorem~4.1]{DaniThomas:divcox}. We note that since there exist wide
Coxeter groups which are not products (for instance the 3-3-3 triangle group),
the following result does not
generalize beyond the right-angled case.

\begin{prop}\label{prop:general_join}
$\mathcal T_0$ is the set of graphs of the form $(\Gamma_1\star\Gamma_2)\star K$, where $\Gamma_1,\Gamma_2$ are aspherical and $K$ is a (possibly empty) clique.
\end{prop}

\begin{proof}  If $\Gamma$ decomposes as in the statement of the proposition, then
$W_{\Gamma}$ decomposes as the product of infinite subgroups
$W_{\Gamma_1}\times (W_{\Gamma_2}\times\mathbb Z_2^{|K|})$, whence
$W_{\Gamma}$ has linear divergence and is therefore wide, i.e.,
$\Gamma\in\mathcal T_0$.  Conversely, suppose that $W_{\Gamma}$ has
linear divergence, and let $\widetilde X_{\Gamma}$ be the universal
cover of the Davis complex (see~\cite{Davis:book}).  Then $\widetilde X_{\Gamma}$ is a CAT(0)
cube complex on which $W_{\Gamma}$ acts properly and cocompactly by
isometries.  Each hyperplane $H$ of $\widetilde X_{\Gamma}$ is
regarded as being labeled by a pair $(v,g)\in\Gamma^{(0)}\times
W_{\Gamma}$, where $gvg^{-1}$ acts as an inversion in the hyperplane
$H$.

Recall that $W_{\Gamma}$ acts \emph{essentially},  in the sense
of~\cite{CapraceSageev:rank_rigidity},  on $\widetilde
X_{\Gamma}$ if for
each hyperplane $H$ the two components of $\widetilde
X_{\Gamma}-H$ each contain points in some $W_{\Gamma}$-orbit which are arbitrarily
far from $H$. A hyperplane which does not have this property is
called \emph{inessential}.

Suppose that the action of $W_{\Gamma}$ on $\widetilde X_{\Gamma}$ is
\emph{essential}.
Then, since $W_{\Gamma}$ is wide, it contains no rank-one isometry of
$\widetilde X_{\Gamma}$ and, hence, the rank-rigidity theorem
of~\cite{CapraceSageev:rank_rigidity} implies that there exist
unbounded convex subcomplexes $\widetilde Y,\widetilde Y'$ such that
$\widetilde X_{\Gamma}=\widetilde Y\times\widetilde Y'$.  It follows
that the link of the vertex in $\widetilde X_{\Gamma}$ decomposes as
the join of aspherical subgraphs.  But this link is exactly $\Gamma$
and hence $\Gamma$ has the desired form.

Now we may assume  $W_{\Gamma}$ is not acting essentially on
$\widetilde X_{\Gamma}$. Thus, by definition, there exists an inessential
hyperplane $H_{(v,1)}$ and it is easy
to see that every generator must commute with $v$.  Indeed, if
$H_{(w,1)}$ and $H_{(v,1)}$ are disjoint hyperplanes, then $\langle
v,w\rangle\{H_{(w,1)}\}$ contains hyperplanes arbitrarily far from
$H_{(v,1)}$ in each of its halfspaces.  Let $K$ be the clique in
$\Gamma$ whose vertices label such inessential hyperplanes.  Then
$\Gamma=\Gamma'\star K$, where $\Gamma'$ is an aspherical set whose
vertices label essential hyperplanes of $\widetilde X_{\Gamma}$.
This provides the desired decomposition of $\Gamma'$ as the join of aspherical
subsets.
\end{proof}


\section{Random right-angled Coxeter groups}\label{sec:random}
We now consider the right-angled Coxeter group $W_{\Gamma}$ where
$\Gamma$ is a random graph in the following sense.  Let $p:\mathbb
N\rightarrow[0,1]$ be a function such that $p(n){n\choose 2}$ has a
limit in $\reals\cup\{\infty\}$ as $n\rightarrow\infty$.  A random
graph on $n$ vertices is formed by declaring each pair of vertices to
span an edge, independently of other pairs, with probability $p=p(n)$.
In other words, we define $G(n,p)$ to be the probability space consisting
of simplicial graphs with $n$ vertices, where, for each graph $\Gamma$
on $n$ vertices, $\prob(\Gamma)=p^E(1-p)^{{n\choose 2}-E}$, where $E$
is the number of edges in $\Gamma$.  This model of random graphs was
introduced by Gilbert in~\cite{Gilbert:Random_graphs},
and is both contemporaneous with and very similar to the
Erd\'{o}s-R\'{e}nyi model of random graphs first studied
in~\cite{ErdosRenyi1,ErdosRenyi2}. For a survey of more recent
results on random graphs see \cite{Chung:tourrandomgraphs}.

Since the assignment $\Gamma\mapsto W_{\Gamma}$ of a finite simplicial
graph to the corresponding right-angled Coxeter group is
bijective~\cite{Muhlherr}, it is sensible to define ``generic''
properties of right-angled Coxeter groups with reference to the above
model of random graphs.  More precisely, if $\pp$ is some
property of right-angled Coxeter groups for which there is a class
$\mathcal G$ of finite simplicial graphs such that $W_{\Gamma}$ has
the property $\pp$ if and only if $\Gamma\in\mathcal G$, then
we say that $W_{\Gamma}$ satisfies $\pp$ \emph{asymptotically
almost surely (a.a.s.)} if $\prob(\Gamma\in\mathcal G\cap
G(n,p))\rightarrow 1$ as $n\rightarrow\infty$. We emphasize that the
notion of asymptotically almost surely depends on the choice of
probability function, $p$, even though it is customary to not
explicitly mention this function in the notation.

The following question describes the author's best guess regarding
the behavior of
thickness and relative hyperbolicity for random right-angled Coxeter
groups. In this section we will provide both theorems and
computations that motivate this picture, but we lead with it to
contextualize the theorems that follow it.

\begin{question}\label{conj:thick_conj}
Let $\mathrm T_m$ be the set of graphs $\Gamma$ for which $W_\Gamma$ is thick of order $m\geq 0$, and denote by $\mathrm T_{\infty}$ the set of graphs for which $W_\Gamma$ is hyperbolic relative to proper subgroups.  Do there exist functions $f_m^-,f_m^+:\mathbb N\rightarrow[0,1],\,m\geq 0,$ such that for all $m\geq 0$, we have $f_m^-=O(f_m^+)$, and $f_{m}^+=O(f_{m-1}^-)$, and  $$\lim_{n\rightarrow\infty}\prob\left(\Gamma\in\mathrm T_m\mid\Gamma\in G(n,p(n)\right)=\begin{cases}
                             0\text{  if  }\frac{p(n)}{f_m^-(n)}\rightarrow 0\\
                             1\text{  if  }\frac{p(n)}{f_m^-(n)}\rightarrow\infty\text{  and  }\frac{p(n)}{f_{m}^+(n)}\rightarrow0

                                                                                                                                                                                                                                                                                                                                               \end{cases}?
$$ Similarly, does there exist $f_{\infty}$ such that $W_{\Gamma}$ is
asymptotically almost surely relatively hyperbolic when
$\Gamma\in G(n,p(n))$ and $p=o(f_{\infty})$?
\end{question}

The situation that would occur in the event of a positive answer to
Question~\ref{conj:thick_conj} is illustrated heuristically in
Figure~\ref{fig:spectrum}.  Given $p_1,p_2\co \naturals\rightarrow[0,1]$,
we place $p_1$ to the left of $p_2$ in the picture of $[0,1]$ if and
only if $p_1=o(p_2)$. Compare also Figure~\ref{fig:results3} which
summarizes the results of this section.

\begin{figure}[h]
      \centering
     \labellist\small\hair 2pt
     \pinlabel {$0$} at 39 11
    \pinlabel {$f_{\infty}$} at 66 11
    \pinlabel {$f_{m}^{-}$} at 128 11
    \pinlabel {$f_{m}^{+}$} at 142 11
    \pinlabel {$f_{m-1}^{-}$} at 170 11
    \pinlabel {$f_{0}^{-}$} at 253 11
    \pinlabel {$f_{0}^{+}$} at 267 11
    \pinlabel {1} at 318 11
    \pinlabel {Rel. hyp.} at 66 60
    \pinlabel {$T_{m}$} at 133 60
    \pinlabel {$T_{m-1}$} at 171 60
    \pinlabel {$T_{0}$} at 258 60
     \pinlabel {Finite or} at 306 70
     \pinlabel {virt. $\mathbb Z$} at 306 60
    \endlabellist
\includegraphics[scale=1]{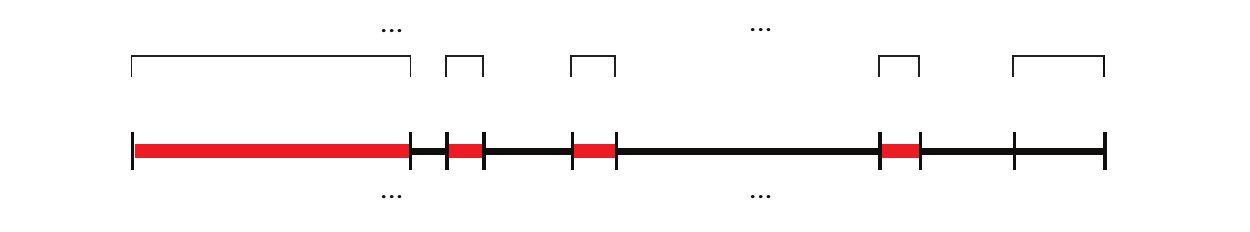}
\caption{Prevalence of thickness along the ``spectrum'' of densities
$p(n)$, if the answer to Question~\ref{conj:thick_conj} is positive;
bold intervals are where, conjecturally, $W_{\Gamma}$ is a.a.s.\ thick of
a specified order.}\label{fig:spectrum}
\end{figure}
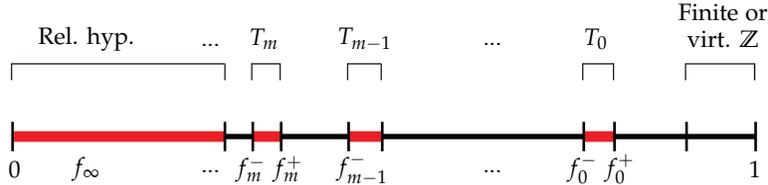

In the interval where $W_{\Gamma}$ is a.a.s.\ relatively
hyperbolic, it is interesting to speculate whether the order of
thickness of the peripheral subgroups might be determined by $p(n)$,
especially in view of Theorem~\ref{thm:generic_relhyp}, which we will
see below.  In other
words, one could sensibly ask if there are functions $g_m^{\pm}$ such
that $W_{\Gamma}$ is a.a.s.\ hyperbolic relative to groups that are
thick of order $n$ for $p$ between $g_m^-$ and $g_m^+$, and if there
is a function $g_{\infty}$ such that $W_{\Gamma}$ is a.a.s.\ hyperbolic
--- i.e.\ hyperbolic relative to hyperbolic subgroups -- when
$p=o(g_{\infty})$.

The results in this section are summarized in
Figure~\ref{fig:results3}.  These results are consistent with a
positive answer to Question~\ref{conj:thick_conj}, but there are
significant ``gaps'' in the spectrum about which nothing is presently
known.


\begin{rmk}[Thickness and connectivity]\label{rem:thick_conn}
If $\Gamma$ is disconnected, then $W_{\Gamma}$ splits as a nontrivial
free product and is therefore not thick.  Hence the function
$f_{\infty}$ from Question~\ref{conj:thick_conj}, if it exists, must
satisfy $\log n/(n f_{\infty})\rightarrow0$, by
Theorem~\ref{thm:generic_relhyp} (as shown in
Figure~\ref{fig:results3}), since $\frac{\log^6n}{n}\rightarrow 0$.  In
other words, there are densities at which $\Gamma$ is a.a.s.\ connected
but $W_{\Gamma}$ is not a.a.s.\ thick.  However, the convergence to 0
of the proportion of random graphs at density $O(\frac{\log n}{n})$ is
quite slow.  This is illustrated in Table~\ref{tab:experiments}, which
shows data selected from the output of many computer
experiments\footnote{Source code available from the authors and at
arXiv.}; for
correctly-chosen $a>0$, even at $n=10000$ it is not yet even clear
that $W_\Gamma$ is not a.a.s.\ thick at density $\frac{a\log
n}{n}$.

\begin{table}[h]\label{tab:experiments}
\begin{minipage}[b]{0.45\textwidth}
\begin{tabular}{|l|c|r|}
\hline
$\mathbf a$ & $\mathbf n$ & \textbf{Prop. thick}\\\hline\hline
 1.95&2000&0.53\\
 1.95&2100&0.515\\
 1.95&4000&0\\\hline
 2&2000&0.8\\
 2&2500&0.46\\
2&3000&0.19\\
2&4000&0.025\\\hline
2.5&2500	&1\\
2.5&3000&0.53\\\hline
\end{tabular}
\end{minipage}
\hspace{-1.4cm}
\begin{minipage}[b]{0.4\textwidth}
\begin{tabular}{|l|c|r|}
\hline
$\mathbf a$ & $\mathbf n$ & \textbf{Prop. thick}\\\hline\hline
2.5&4000&0\\\hline
3&4000&0.5\\
3&5000&0\\\hline
4&4000&1\\
4&10000&1\\\hline
5&4000&1\\
5&10000&1\\\hline
10&4000&1\\
10&10000&1\\\hline
\end{tabular}
\end{minipage}
\caption{Experimental proportion of $\Gamma\in G\left(\mathbf n,\frac{\mathbf a\log\mathbf n}{\mathbf n}\right)$ that are thick.  For each $\mathbf a$, this proportion tends to 0 as $\mathbf n\rightarrow\infty$ by Theorem~\ref{thm:generic_relhyp} but, as illustrated, may do so quite slowly.}
\end{table}

\end{rmk}

\subsection{Behaviour at low densities}\label{subsec:basic_statements}

We collect a few facts about random right-angled Coxeter groups:

\begin{thm}\label{thm:free_prod}
$W_{\Gamma}$ asymptotically almost surely decomposes as a nontrivial free product, if and only if there exists $\epsilon>0$ such that $p(n)<\frac{(1-\epsilon)\log n}{n}$.  Hence, if $p(n)<\frac{(1-\epsilon)\log n}{n}$, then the divergence of $W(\Gamma)$ is a.a.s.\ infinite.

If there exists $\epsilon>0$ such that $p(n)>\frac{(1+\epsilon)\log n}{n}$ and $k\in\mathbb N$ such that $n^kp(n)^{k^2}\rightarrow 0$, then a.a.s.\ $\Gamma$ has no separating clique, and hence $W_{\Gamma}$ is a.a.s.\ one-ended and has finite divergence function.
\end{thm}

\begin{proof}
$W_{\Gamma}$ admits a nontrivial free product decomposition if and only if $\Gamma$ is disconnected, and $\log n/n$ is the threshold for $p(n)$ above which connectedness occurs a.a.s.\ and below which disconnectedness occurs a.a.s.\ (see~\cite{ErdosRenyi2}).

Let $K_n=K_n(\Gamma)$ equal 1 or 0 according to whether $\Gamma$ is disconnected.  For $0\leq j\leq n$, let $K_n^j(\Gamma)=\sum_{\Lambda}K_{n-j}(\Gamma-\Lambda)$, where $\Lambda$ varies over the size-$j$ subgraphs of $\Gamma$.  Then $\expect(K_n^j)={n\choose j}\expect(K_{n-j})p^{j\choose 2}$ is an upper bound for the expected number of separating $j$-simplices, and the expected number of separating simplices in $\Gamma$ is therefore bounded by $$\sum_{j=0}^{n-2}{n\choose j}\expect(K_{n-j})p^{j\choose 2}.$$
Now, for $p(n)>(1+\epsilon)\frac{\log(n)}{n},$ Theorem~1 of~\cite{ErdosRenyi1} implies that $\sum_{j\leq k}{n\choose j}\expect(K_{n-j})p^{j\choose 2}$ tends to 0 for any fixed $k$.  If $p(n)$ is sufficiently small to ensure that a.a.s.\ all cliques in $\Gamma$ have size $O(1)$, i.e.\ if there exists $k$ such that ${n\choose k}p^{k\choose 2}\rightarrow 0$, then the preceding sum bounds the limiting expected number of separating cliques of any size, and the proof is complete.
\end{proof}

\begin{thm}\label{thm:lin_quad}
If $p(n)=o\left(n^{\frac{-n}{2(n-2)}}\right)$, then $W_{\Gamma}$ is not thick of order 0, and hence has at least quadratic divergence, a.a.s.
\end{thm}

\begin{proof}
$W_{\Gamma}$ is thick of order 0 only if $\Gamma$ admits a nontrivial
join decomposition in which each factor has at least two vertices, by
Proposition~\ref{prop:general_join}.  Hence $W_{\Gamma}$ is thick of
order 0 only if there exists $a\in\mathbb N$ with $2\leq a\leq n-2$
such that $K_{a,n-a}$ spans $\Gamma$.  In~\cite{ErdosRenyi2}, it is
shown that, for each such $a$, there is no such subgraph,
asymptotically almost surely, if the number $N$ of edges in $\Gamma$
satisfies $$N=o\left(n^{2-\frac{n}{a(n-a)}}\right),$$ where $\Gamma$
is a random graph in the slightly different model considered in that
paper.

The same conclusion applies in the present situation provided
the expected number $\expect(N)=p(n){n\choose 2}$ of edges tends with
$n$ to infinity, by~\cite[Theorem~2.2]{Bollobas:book}).  It follows
that if $p(n)=o\left(n^{-\frac{n}{2(n-2)}}\right)$ and
$p(n)n^2\rightarrow\infty$, then
$N=o\left(n^{2-\frac{n}{2(n-2)}}\right)$ and hence $\Gamma$ does not
contain $K_{a,n-a}$, with $2\leq a\leq n-2$, a.a.s.\ In this case, we
thus have $W_{\Gamma}$ is not thick of order $0$, and hence has
superlinear divergence. By
\cite[Corollary~B]{CapraceSageev:rank_rigidity}, since $W_{\Gamma}$
acts co-compactly on its Davis complex it contains a periodic
rank-one geodesic and thus by \cite[Proposition~3.3]{KapovichLeeb:3manifolds} the
divergence of $W_{\Gamma}$ is at least quadratic.

If $\expect(N)$
does not tend with $n$ to infinity, then $p(n)n^2$ is bounded, whence
$p(n)$ grows slowly enough to ensure that $\Gamma$ is a.a.s.\
disconnected, and hence $W_{\Gamma}$ has infinite divergence.
\end{proof}

\begin{thm}\label{thm:generic_relhyp}
If $p(n)n\rightarrow\infty$ and $p(n)^6n^5\rightarrow0$, then the
following holds asymptotically almost surely: $\Gamma$ has a component
$\Gamma'$ such that $W_{\Gamma'}$ is hyperbolic relative to
a nonempty collection of proper subgroups, each isomorphic to
$D_{\infty}\times D_{\infty}$.  Hence $W_{\Gamma}$ is a.a.s.
hyperbolic relative to a nonempty collection of proper
$D_{\infty}\times D_{\infty}$ subgroups, at least one of which is not
a proper free factor of $W_{\Gamma}$.
\end{thm}

\begin{rmk}
Of greatest interest are densities $p(n)$ growing faster than
$\frac{\log n}{n}$ but slower than $n^{-1/6}$.
At such densities, Theorem~\ref{thm:free_prod} and
Theorem~\ref{thm:generic_relhyp}
together ensure that $W_{\Gamma}$ is asymptotically
almost surely one-ended and hyperbolic relative to $D_{\infty}\times
D_{\infty}$ subgroups.
\end{rmk}

\begin{proof}[Proof of Theorem~\ref{thm:generic_relhyp}]
Since $pn\rightarrow\infty$,~\cite{ErdosRenyi3} together
with~\cite[Theorem~2.2.(ii)]{Bollobas:book} implies that a.a.s.
$\Gamma$ has a \emph{giant component} $\Gamma'$ containing a positive
proportion $\alpha\in(0,1)$ of the vertices, and every other component
$\Gamma_i$ has no more than $O(\log n)$ vertices.  It suffices to show
that, a.a.s, $\Gamma'$ contains $K_{2,2}$ as an induced proper
subgraph and $\Gamma$ does not contain $K_{2,3}$.  Indeed, the second
assertion, together with Lemma~\ref{lem:5_vertices} implies that every
element of $\mathcal T$ arising as an induced subgraph of $\Gamma'$ is
isomorphic to $K_{2,2}$.  The first assertion, together with
Theorem~\ref{thm:thick_rel_hyp}, will then complete the proof.

\textbf{$K_{2,3}$ is a.a.s.\ absent:}
Since $p(n)^6n^5\rightarrow0$ as
$n\rightarrow\infty$ by hypothesis,
Corollary~5 of~\cite{ErdosRenyi2}
implies that, a.a.s., $\Gamma$, and therefore $\Gamma'$, does not
contain $K_{2,3}$.

\textbf{An induced $K_{2,2}$ a.a.s.\ appears in $\Gamma'$:}  Let $v_1,\ldots,v_4$ be distinct vertices in the random size-$n$ graph $\Gamma$, and let the random variable $I(v_1,\ldots,v_4)$ take the value 1 or 0 according to whether or not $\{v_1,\ldots,v_4\}$ is the vertex set of an induced $K_{2,2}$ in $\Gamma$.  The random variable $S_n=\sum_{v_1,v_2,v_3,v_4}I(v_1,\ldots,v_4)$ counts each induced $K_{2,2}$ in $\Gamma$ eight times, reflecting the eight automorphisms of $K_{2,2}$.  Since there are ${n\choose 4}$ such quadruples, and each forms an induced copy of $K_{2,2}$ exactly when there is some permutation $\sigma:\{1,2,3,4\}\rightarrow\{1,2,3,4\}$ such that $v_{\sigma(i)}$ is adjacent to $v_{\sigma(i)+1}$ for each $i$, and the remaining two possible edges are absent, we have $\expect(S_4)=24{n\choose 4}p^4(1-p)^2$.

Let $N\in\mathbb N$ and let $\epsilon\in(0,1)$.  The preceding discussion shows that since $p(n)n\rightarrow\infty$, there exists $N_1\in\mathbb N$ such that $\expect(S_n)\geq\frac{N}{\epsilon}$ for all $n\geq N_1$.  The proof of Theorem~4.1 of~\cite{CharneyFarber} shows that, since $pn\rightarrow\infty$ and $(1-p)n^2\rightarrow\infty$, $$\frac{\expect(S_n)^2}{\expect(S_n^2)}\rightarrow 1,$$
so that there exists $N_2\in\mathbb N$ such that $$\frac{\expect(S_n)^2}{\expect(S_n^2)}>1-\epsilon$$ for $n\geq N_2$.  The Paley-Zygmund inequality implies that for all $n\geq\max\{N_1,N_2\}$,
\begin{eqnarray*}
 \prob(S_n\geq N)&\geq&\prob(S_n\geq\epsilon\expect(S_n))\\
 &\geq&(1-\epsilon)^2\frac{\expect(S_n)^2}{\expect(S_n^2)}>(1-\epsilon)^3.
\end{eqnarray*}

This implies that for each $N\in\mathbb N$, we have $\lim_n\prob(S_n<N)=0$.  Lemma~\ref{lem:unicyclics} below states that a.a.s., every component of $\Gamma$ is either a tree or equal to $\Gamma'$.  Hence $\prob(S'_n<16)\rightarrow0$ as $n\rightarrow\infty$, where $S'_n$ counts the squares (ignoring symmetry) in $\Gamma'$.  Thus $\Gamma'$ a.a.s.\ contains at least two induced copies of $K_{2,2}$.
\end{proof}

\begin{rmk}
The fact that $W_{\Gamma}$ is hyperbolic relative to $D_{\infty}\times D_{\infty}$ subgroups that are not free factors can be seen slightly more easily, by first producing induced $K_{2,2}$ subgraphs in $\Gamma$ and verifying that $\Gamma$ a.a.s.\ does not contain $K_{2,3}$, as in the proof of Theorem~\ref{thm:generic_relhyp}, and then observing that by Theorem~5.16 of~\cite{Bollobas:book}, $\Gamma$ a.a.s.\ has no component which is a 4-cycle.  Theorem~\ref{thm:generic_relhyp} is of course a stronger conclusion, since it rules out the possibility that $W_{\Gamma'}$ is hyperbolic and every 4-cycle lies in a unicyclic component that is not a 4-cycle.
\end{rmk}

\begin{lem}\label{lem:unicyclics}
Let $\Gamma\in G(n,p(n))$, with $p(n)$ satisfying the hypotheses of Theorem~\ref{thm:generic_relhyp}.  Asymptotically almost surely, each component of $\Gamma$ is either the giant component or is a tree.
\end{lem}

\begin{proof}[Proof of Lemma~\ref{lem:unicyclics}]
This follows immediately from~\cite[Theorem~6.10.(iii)]{Bollobas:book} and Theorem~\cite[Theorem~2.2.(ii)]{Bollobas:book}.
\end{proof}

\begin{lem}\label{lem:5_vertices}
If $\Lambda\in\mathcal T$, then either $\Lambda\cong K_{2,2}$ or $\Lambda$ contains $K_{2,3}$.
\end{lem}

\begin{proof}
Since $\Lambda$ must contain the join of two subgraphs of diameter at least 2, $|\Lambda^0|\geq 4$ and either $\Lambda\cong K_{2,2}$ or $|\Lambda|\geq 5$.  In the latter case, suppose that each maximal join in $\Lambda$ is isomorphic to $K_{2,2}$ and let $\Lambda_0\subset\Lambda$ be such a join.  Then no two non-adjacent vertices in $\Lambda_0$ have a common adjacent vertex, since otherwise $\Lambda_0$ would extend to a copy of $K_{2,3}$.  Hence $\Lambda\cong K_{2,2}$, a contradiction.
\end{proof}

\subsection{Behavior at high densities}\label{subsec:high_density}
Charney-Farber showed in~\cite{CharneyFarber} that a random right-angled Coxeter group on $n$ vertices is a.a.s.\ finite when $(1-p(n))n^2\rightarrow 0$ as $n\rightarrow\infty$.  The following description of random right-angled Coxeter groups for rapidly-growing $p(n)$ generalizes this result.

\begin{thm}\label{thm:high_density}
Suppose $(1-p(n))n^2\rightarrow\alpha$ as $n\rightarrow\infty$, for
some $\alpha\in[0,\infty)$ and let the random variable $M_n$ count the
number of ``missing edges'' in $\Gamma\in\mathcal G(n,p)$, i.e.\ the
number of pairs of distinct vertices that are not joined by an edge.
Then $M_n=O(1)$ a.a.s.\ and:
\begin{enumerate}
 \item With probability tending to $e^{-\alpha/2}$, $M_n=0$ and the group $W_{\Gamma}$ is finite.
 \item With probability tending to $\frac{\alpha}{2}e^{-\alpha/2}$, $M_n=1$ and the group $W_{\Gamma}$ is virtually $\mathbb Z$ and thus hyperbolic.
 \item With probability tending to $1-(1+\frac{\alpha}{2})e^{-\alpha/2}$, $M_n\geq 2$ and the group $W_{\Gamma}$ is virtually $\mathbb Z^{M_n}$, and is thus thick of order 0 and has linear divergence.
\end{enumerate}
\end{thm}

\begin{proof}
\textbf{Finite and virtually $\mathbb Z$:}  If $M_n=0$, then $\Gamma$ is a complete graph, so that $W_{\Gamma}\cong\mathbb Z_2^n$ is finite.  Conversely, if $W_{\Gamma}$ is finite, then since any two nonadjacent vertices together generate a subgroup isomorphic to $D_{\infty}$, we see that $M_n=0$.  Similarly, $W_{\Gamma}$ is virtually $\mathbb Z$ if and only if $M_n=1$.

For $k\geq 0$, we have $$\prob(M_n=k)={{n\choose 2}\choose k}(1-p(n))^kp^{{n\choose 2}-k},$$ and $$p(n)^{{n\choose 2}-k}\sim e^{-\alpha/2}.$$  Hence $\prob(M_n=0)\rightarrow e^{-\alpha/2}$ while $\prob(M_n=1)\sim{n\choose 2}\left(\frac{\alpha}{n^2}\right)e^{-\alpha/2}\rightarrow\frac{-\alpha}{2}e^{-\alpha/2}$.  This establishes the first two assertions.

\textbf{Thick of order 0:}  For each vertex $v\in\Gamma$, let $I_v$ be 1 or 0 according to whether or not $v$ belongs to exactly one missing edge, so that $\prob(I_v=1)=\expect(I_v)=n(1-p(n))p(n)^{n-2}$.  Let $E_n=\sum_vI_v$ count the number of vertices belonging to exactly one missing edge, and observe that $\expect(E_n)=n^2(1-p(n))p(n)^{n-2}\sim\alpha$.

Similarly, let $J_v$ be 1 or 0 according to whether or not $v$ belongs to at least one missing edge, and let $F_n=\sum_vJ_v$ count the vertices appearing in at least one missing edge.  Note that $\prob(J_v=1)=\expect(J_v)=1-p(n)^{n-1}$.  Hence
\begin{eqnarray*}
 \expect(F_n)&=&n(1-p(n)^{n-1})\\
 &=&n\left[1-\left(1-\frac{\alpha}{n^2}\right)^{n-1}\right]\\
 &=&\frac{\alpha n(n-1)}{n^2}+o(1)\sim\alpha.
\end{eqnarray*}

Since $F_n\geq E_n$, and $\expect(F_n-E_n)\rightarrow 0$, a.a.s.\
$F_n=E_n$.  In other words, a.a.s.\ every vertex occurs in at most one
missing edge.  Therefore, a.a.s.\ there are pairwise-distinct vertices
$v_1,\ldots,v_k,w_1,\ldots,w_k$ such that $v_i,w_i$ are not adjacent
for all $i$ and every other pair of vertices spans an edge.  This
implies that $W_{\Gamma}$ is virtually the product of $k$ copies of
$D_{\infty}$.

The above argument shows that, a.a.s.\ $M_n=\frac{E_n}{2}$.  For distinct vertices $v,w$, we have $$\prob(I_vI_w=1)=(n-1)^2p^{2n-5}(1-p)^2+p^{2n-4}(1-p),$$
from which a computation shows that $\expect(M_n)\rightarrow\frac{\alpha(\alpha+1)}{8}$.  It follows from Markov's inequality that $M_n=O(1)$ a.a.s.
\end{proof}

\subsection{Constant-density behavior}\label{sec:constant_density}
In this section, we prove:

\begin{thm}\label{thm:constant_density}
For $\Gamma\in G(n,\frac{1}{2})$, the group $W_{\Gamma}$ is a.a.s.\ thick.
\end{thm}

The following lemma isolates the most crucial estimates we will use in the
proof of the
theorem.

\begin{lem}\label{lem:estimates}
Let $\pi_n=\prob(\Gamma\not\in\mathcal T|\Gamma\in G(n,\frac{1}{2}))$.  Then:
\begin{enumerate}
\item $\pi_{2n}\leq\pi_n^2+f(n)$, where $f(n)=2n\sum_{i=0}^n{n \choose i}2^{-n-{i \choose 2}}$.\label{item:split}
\item $\pi_{2n}\leq\pi_n^2+2\pi_n(1-\pi_n)\frac{nc(n)}{2^nt(n)}+(1-\pi_n)^2$, where $c(n)$ is the number of cliques in the disjoint union of all $\mathcal T$-graphs on $n$ vertices, and $t(n)$ is the number of such graphs.\label{item:cond_est}
\item \label{item:addone}$\pi_{n+1}\leq\pi_n+f(n)$.
\end{enumerate}
\end{lem}

\begin{proof}
Let $\Gamma\in G(2n,\frac{1}{2})$ and let $A\sqcup B$ be a partition of $\Gamma^{(0)}$ into sets of size $n$.  For $v\in B$, we denote by $\link_A(v)$ the set of vertices in $A$ adjacent to $v$.  Note that if $\Gamma\not\in\mathcal T$, then one of the following holds:
\begin{enumerate}[(i)]
 \item \label{item:bothrelhyp} The subgraphs generated by $A,B$ are not in $\mathcal T$.
 \item \label{item:clique} There exists $v\in B$ [or $v\in A$] such that $\link_A(v)$ [or $\link_B(v)$] is a (possibly empty) clique.
 \end{enumerate}
To establish this dichotomy, first we assume~\eqref{item:bothrelhyp}
does
not hold, and hence, without loss of generality we may assume
the subgraph generated by
$A$ is in $\mathcal T$.  If additionally~\eqref{item:clique} does not
hold we show this yields $\Gamma\in\mathcal T$ which is a
contradiction. Condition~\eqref{item:clique} implies that
for each vertex $v$ of $B$
the set $\link_A(v)$ is nonempty and has diameter exceeding 1.
Now, for each $v\in B$ we have that the subgraph $\Gamma_v$ of $\Gamma$
generated by $A\cup\{v\}\in\mathcal T$ is
in $\mathcal T$ since it is obtained by coning off a set of diameter at
least 2 and applying
Definition~\ref{defn:thick_graph}(2).
Also, for each $v,v'\in B$, since the graphs $\Gamma_v$ and
$\Gamma_{v'}$ are both thick and their intersection is the thick
graph generated by $A$, we see that the graph generated by
$A\cup\{v,v'\}$ which is the generalized union of $\Gamma_v$  and
$\Gamma_{v'}$ and is thus thick by  Definition~\ref{defn:thick_graph}(3).
Thus, by adding one vertex from $B$ at a time in the above way we see
that $\Gamma\in\mathcal T$.

Next, we claim that $\prob((\ref{item:bothrelhyp}))=\pi_n^2$.  Indeed, since in
the construction of $\Gamma$, edges joining pairs of vertices in $A$
are added independently of those joining vertices in $B$, the events
``$A$ generates a subgraph in $\mathcal T$'' and ``$B$ generates a
subgraph in $\mathcal T$'' are independent.  Moreover, the subgraphs
of $\Gamma$ generated by $A$ and $B$ are in $G(n,\frac{1}{2})$.  It
follows that~\eqref{item:bothrelhyp} occurs with probability
$\pi_n^2$, whence $$\pi_{2n}\leq\pi_n^2+\prob((\ref{item:clique})).$$

We finally show that $\prob(\mathrm{(\ref{item:clique})})\leq f(n)$.
To this end, let $\mathcal V$ be the number of vertices of $B$ whose
links in $A$ are (possibly empty) cliques.  Then
$\prob(\mathrm{(\ref{item:clique})})\leq2$ and $\prob(\mathcal V>0)\leq2\expect(\mathcal
V)$.  The initial factor of $2$ reflects the fact that we are assuming
that $A\in\mathcal T$ and counting vertices in $B$ whose links in $A$
are cliques; $(\ref{item:clique})$ could just as easily occur with the roles
of $A,B$ reversed.

For each $v\in B$, if $\link_A(v)$ has $k$ vertices, then it is
generated by one of ${n\choose k}$ subsets of $A$.  Each such subset
is a clique with probability $2^{-{k\choose 2}}$, and such a subset
generates $\link_A(v)$ with probability $2^{-k}2^{k-n}=2^{-n}$,
reflecting the fact that the $k$ vertices of the putative link must be
adjacent to $v$ and the $n-k$ remaining vertices of $A$ must not.
Summing over $k$ yields the probability that $\link_A(v)$ is a clique,
so that $\expect(\mathcal V)=n\sum_{k=0}^n{n\choose k}2^{-n-{k\choose
2}}$, and Claim~\eqref{item:split} follows.

To establish claim~\eqref{item:cond_est}, write $\Gamma^{(0)}=A\sqcup B$ as above.  If $\Gamma\not\in\mathcal T$, then one of the following holds:
\begin{enumerate}
 \item the subgraphs generated by $A,B$ are both not in $\mathcal T$. This event occurs with probability $\pi_n^2$.
 \item Exactly one of the subgraphs generated by $A,B$ belongs to $\mathcal T$.  In this case, suppose that $A$ generates a subgraph in $\mathcal T$.  This subgraph is among the $t(n)$ graphs of its size in $\mathcal T$, and as above, $B$ must contain a vertex $v$ whose link in $A$ generates one of the $c(n)$ possible cliques.  There are $n$ choices for this vertex, and each has a given clique as its link with probability at most $2^{-n}$.  Hence this situation occurs with probability at most $2\pi_n(1-\pi_n)nc(n)2^{-n}t(n)^{-1}$.
 \item The subgraphs generated by $A,B$ both belong to $\mathcal T$.  In this case, it must be true that some vertex in $A$ has link in $B$ a clique (or vice versa), but we do not use this fact; we just note that the probability of this event is certainly at most $(1-\pi_n)^2$.
\end{enumerate}

Finally, to establish Claim~\eqref{item:addone}, regard the size-$(n+1)$ graph $\Gamma$ as the subgraph of $\Gamma$ generated by $A\sqcup\{v\}$, with $v$ a vertex.  If $\Gamma\not\in\mathcal T$, then either $A\not\in\mathcal T$ or the link of $v$ is a clique.  The claim now follows by arguing as in the proof of Claim~\eqref{item:split}.
\end{proof}

\begin{rmk} The relation between the first two parts of the
    above lemma are as follows.
In the language of conditional probability, to prove
Lemma~\ref{lem:estimates}\eqref{item:split} we use the
fact that:
$$\pi_{2n}\leq \prob[A,B\notin \mathcal
T]+\prob[\eqref{item:clique}].$$

Whereas, for Lemma~\ref{lem:estimates}\eqref{item:cond_est} we exploited
the following:
$$\pi_{2n}\leq \prob[A,B\notin \mathcal T]+2\prob[A\in \mathcal T,
B\notin \mathcal T]\cdot\prob[\eqref{item:clique}_B| A\in \mathcal T,
B\notin \mathcal T]+ \prob[A,B\in \mathcal T],$$
where $\eqref{item:clique}_B$ is the same as $\eqref{item:clique}$
except that we require only the condition on links of vertices of $B$.
We then sum over these probabilities to yield Lemma~\ref{lem:estimates}\eqref{item:cond_est}.
\end{rmk}

We will make use of the following estimate:

\begin{lem}\label{lem:hoeffding}
Let $X_n$ be a binomial random variable with mean $\frac{1}{2}\cdot n$ and variance $\frac{1}{4}\cdot n$.  Then for all $M\leq\frac{n}{2}$, we have $$\prob\left(X_n\leq M\right)\leq\exp\left(-\frac{n}{2}+2M-\frac{2M^2}{2}\right).$$
\end{lem}

\begin{proof}
Viewing $X_n$ as the sum of $n$ Bernoulli trials, this follows from Hoeffding's inequality~\cite{Hoeffding}.
\end{proof}

\begin{lem}\label{lem:fsmall}
The function $f$ of Lemma~\ref{lem:estimates} has the following properties:
\begin{enumerate}
 \item  $f(n)\stackrel{n}{\longrightarrow}0$ exponentially and, in particular, $\sum_{n\geq0}f(n)<\infty$.
 \item $f(n)<0.03760$ for all $n\geq 18$.
\end{enumerate}

\end{lem}

\begin{proof}
Let $M=\lfloor n^{a/b}\rfloor$ for natural numbers $a<b$, and write
\begin{eqnarray*}
f(n)&=&2n\left[\sum_{i=0}^M{n\choose i}2^{-n-{i\choose 2}}+\sum_{i=M+1}^n{n\choose i}2^{-n-{i\choose 2}}\right]\\
&=&2n\cdot\mathrm{(I)}+2n\cdot\mathrm{(II)}.
\end{eqnarray*}
For each $n$, $$\mathrm{(I)}\leq2^{-n}\sum_{i=0}^M{n\choose i}=\prob(X_n\leq M),$$
where $X_n$ is a binomial random variable with mean $n\cdot\frac{1}{2}$.  From Lemma~\ref{lem:hoeffding}, we have, for $M\leq n/2$,
\begin{eqnarray*}
\mathrm{(I)}&\leq&\exp\left[-\frac{n}{2}+2M-\frac{2M^2}{n}\right]\\
&\leq&e^{-n/2}e^{2\lfloor n^{a/b}\rfloor}e^{-2\lfloor n^{a/b}\rfloor^2/n}:=g(n,M)
\end{eqnarray*}

We also have:
\begin{eqnarray*}
\mathrm{(II)}&\leq& 2^{-n-{M\choose 2}}\sum_{i=M+1}^n{n\choose i}\\
&\leq&2^{-{M+1\choose 2}}\leq2^{-n^{a/b}(n^{a/b}-1)/2}.
\end{eqnarray*}
Suppose now that $a,b$ also satisfy $2a/b>1$. Then the lemma follows from summing the above estimates: $f(n)$ decays exponentially and is hence summable.  This establishes the first assertion.

The second assertion requires a refinement of one of the above bounds.  Let $a=2,b=3$, and let $M=\lfloor n^{a/b}\rfloor, X_n,$ and the expressions $\mathrm{(I)}$ and $\mathrm{(II)}$ be as above.  As before, we have $$\mathrm{(II)}\leq2^{-n^{2/3}(n^{2/3}-1)/2}.$$  We need to estimate $\mathrm{(I)}$ more carefully when $n\geq 18$. We thus write:
\begin{eqnarray*}
\mathrm{(I)}&\leq&2^{-n}\left(\sum_{i=0}^5{n \choose i}2^{-{i\choose 2}}\right)+2^{-{6\choose 2}}\prob(X_n\leq \lfloor n^{2/3}\rfloor)\\
&\leq&2^{-n}\left(\sum_{i=0}^5{n \choose i}2^{-{i\choose 2}}\right)+2^{-{6\choose 2}}g(n,\lfloor n^{2/3}\rfloor):=h(n).\\
\end{eqnarray*}
The second inequality is an application of Lemma~\ref{lem:hoeffding}, justified by the fact that $n^{2/3}<n/2$ for $n\geq 18$.  Hence $$f(n)\leq 2nh(n)+2n\cdot2^{-n^{2/3}(n^{2/3}-1)/2}.$$  The second term is strictly decreasing for $n\geq 8$, as can be seen by differentiating, and takes a value less than $3.09\cdot10^{-5}$ at $n=18$.  Next, a straightforward computation gives $$g(n,\lfloor n^{2/3}\rfloor)\leq\exp\left(-\frac{n}{2}+2n^{2/3}-2n^{1/3}+4n^{-1/3}-\frac{2}{n}\right),$$
which is decreasing for $n\geq 12$ and, for $n=18$, yields $$2n\cdot2^{-{6\choose 2}}\cdot g(n,\lfloor n^{2/3}\rfloor)\leq0.00273.$$ The remaining term can be shown by direct differentiation to decrease for $n\geq5$, and takes the value $0.3484$ at $n=18$.  Combining the above shows that $f(n)\leq3.09\cdot10^{-5}+0.00273+0.03484=0.03760$ for $n\geq18$.
\end{proof}

\begin{rmk}
The bound provided by Lemma~\ref{lem:fsmall}.(2) is somewhat crude, since in fact $f(18)\approx0.00101$.  However, as we will see in the proof of Theorem~\ref{thm:constant_density}, any bound sharper than around $f(18)\leq0.06045$ is sufficient.
\end{rmk}

\begin{proof}[Proof of Theorem~\ref{thm:constant_density}] The idea of
the proof is to use Lemma~\ref{lem:estimates}.\eqref{item:split}
and the fact that $f$ is small to get convergence to $0$ of a
subsequence of $(\pi_n)$. Then, we use this in order to show that $(\pi_n)$
converges to $0$, and then apply
Lemma~\ref{lem:estimates}.\eqref{item:addone} and the summability of
$f$.

{\bf Accumulation at $0$ implies convergence to $0$.} For each $n,k$,
Lemma~\ref{lem:estimates}.\eqref{item:addone} yields: $$\pi_{n+k}\leq\pi_n+\sum_{i=0}^{k-1}f(i+n)<\pi_n+\sum_{i=n}^{\infty}f(i).$$  Suppose that $0$ is an accumulation point of $(\pi_n)$.  Then for each $\epsilon>0$, we can choose $n$ so that $\pi_n<\frac{\epsilon}{2}$ and $\sum_{i=n}^{\infty}f(n)<\frac{\epsilon}{2}$.  The latter inequality follows from summability of $f$, i.e.\ from Lemma~\ref{lem:fsmall}.(1).  Hence for all $k$, we have $\pi_{n+k}<\epsilon$, i.e.\ $\pi_n\stackrel{n}{\longrightarrow}0$.


{\bf Non-accumulation at $0$ implies convergence to $1$.} Suppose now that the subsequence $(\pi_{k\cdot 2^m})_{m\in\naturals}$ does not have $0$ as an accumulation point for some $k\in\naturals$. Then we claim that $(\pi_{k\cdot 2^m})$ converges to $1$. Indeed, consider the smallest accumulation point $\pi$ of the sequence, and suppose that it is the limit of the subsequence $(\pi_{k\cdot 2^{m_i}})_{i\in\naturals}$. We have to show $\pi=1$. By Lemma~\ref{lem:estimates}.\eqref{item:split} and the fact that $f$ converges to $0$, we get that any accumulation point $\pi'$ of $(\pi_{k\cdot 2^{m_i+1}})$ satisfies $\pi'\leq \pi^2$. As we also have $\pi\leq \pi'$, we get $\pi\leq\pi^2$, so that $\pi=1$.

{\bf A subsequence bounded away from $1$.} It is thus sufficient to
show that the subsequence $(\pi_{k\cdot 2^m})_{m\in\naturals}$ is
bounded away from $1$ for some $k\in\naturals$. In fact, if this is
the case then $(\pi_{k\cdot 2^m})_{m\in\naturals}$ does not converge to $1$, hence it must have $0$ as an accumulation point, and hence $(\pi_n)$ converges to $0$ as required.  Suppose that for some $k$, we have $m_0\in\naturals$ and constants $\alpha,\beta\in[0,1)$ such that $f(k\cdot 2^m)\leq\beta$ for all $m\geq m_0$ and $\pi_{k\cdot2^{m_0}}\leq\alpha$.  Suppose, moreover, that $\alpha^2+\beta<\alpha$.  Then $\pi_{k\cdot2^{m_0+1}}<\alpha$ by Lemma~\ref{lem:estimates}.\eqref{item:split}, and by induction and the same lemma we have $\pi_{k\cdot2^{m}}<\alpha$ for all $m\geq m_0$.

Let $k=9,m_0=1$.  The computer program in Subsection~\ref{subsec:code} returned the data:
\begin{itemize}
 \item $t(9)=14853635863$;
 \item $c(9)=683846354560$;
 \item $\pi_9=1-t(9)/2^{9\choose 2}\approx0.78385$,
\end{itemize}
together with which Lemma~\ref{lem:estimates}.\eqref{item:cond_est} implies
$$\pi_{18}\leq\alpha:=\left(1-\frac{t(9)}{2^{36}}\right)^2+\frac{t(9)^2}{2^{36}}+2\left(1-\frac{t(9)}{2^{36}}\right)\cdot\frac{t(9)^2}{2^{36}}\cdot\frac{9\cdot c(9)}{512\cdot t(9)}\approx 0.93537.$$  Lemma~\ref{lem:fsmall}.(2) gives $f(n)\leq\beta=0.03760$ for all $n\geq 18$.  The above discussion, together with the fact that these values satisfy $\alpha^2+\beta<\alpha$, implies that $(\pi_{9\cdot 2^m})$ is bounded away from $1$, whence $\pi_n\stackrel{n}{\longrightarrow}0$, i.e. $\Gamma$ is a.a.s.\ in $\mathcal T$.
\end{proof}

\section{Detecting thickness algorithmically}\label{sec:algorithm}
In this section, we exhibit a polynomial-time algorithm for deciding
whether a finite graph is in $\mathcal T$.  The construction of the
algorithm presented in this section prioritized simplicity over speed.
We also provide a C++ implementation of a simple algorithm to compute
the constants needed in the proof of
Theorem~\ref{thm:constant_density}.  The main part of this computer
program implements the algorithm for deciding if a given right-angled
Coxeter group is thick.

\begin{thm}\label{thm:algorithm}
There exists an algorithm which decides, in polynomial time, whether
a graph $\Gamma$ is in $\mathcal T$.  Hence the problem of deciding whether a right-angled Coxeter group admits a relatively hyperbolic structure is soluble in polynomial time.
\end{thm}

\begin{proof}
The second assertion follows from the first by Theorem~\ref{thm:thick_rel_hyp}.  The algorithm takes as input the finite simplicial graph $\Gamma$ on $n$ vertices and decides whether $\Gamma\in\mathcal T$. For ease of exposition, we provide an algorithm which admits an easy description, but we note that there are more efficient algorithms; in particular the code in Section~\ref{subsec:code} contains an implementation of a more efficient algorithm for the same task.  The steps are:

\begin{enumerate}
 \item Make a list $\mathcal M$ of all induced $K_{2,2}$ subgraphs of $\Gamma$.  The running time is in $O(n^4)$ and $|\mathcal M|$ is in $O(n^4)$.
 \item Make a list $\mathcal N $ of pairs of non-adjacent vertices.  The running time is in $O(n^2)$ and $|\mathcal N|$ is in $O(n^2)$.
 \item Perform a \emph{union subroutine}, i.e. for each pair $M,M'\in\mathcal M$, determine whether $M\cap M'$ contains some $(v,v')\in\mathcal N$.  If so, modify $\mathcal M$ by removing $M$ and $M'$ adding the subgraph induced by $M\cup M'$.  The running time of a union subroutine is in $O(n^{11})$.
 \item Perform a \emph{coning subroutine}, i.e. for each $M\in\mathcal M$ and each vertex $v$, determine whether there exists $(w,w')\in\mathcal N$ such that $w,w'\in M$ and both adjacent to $v$.  If so, replace $M$ by the subgraph generated by $M\cup\{v\}$.  The running time of a coning subroutine is in $O(n^7)$.
 \item If $\mathcal M$ did not change during the coning and union subroutines, then we are finished: the graph is thick if and only if $|\mathcal M|=1$ and the unique element of $\mathcal M$ is $\Gamma$.
 \item If $\mathcal M$ changed, then return to Step~(2).
\end{enumerate}

The number of union subroutines that modify $\mathcal M$ is in $O(n^4)$ since each such union subroutine decreases $|\mathcal M|$.  The number of coning subroutines that modify $\mathcal M$ is in $O(n^5)$ since each such subroutine increases the size of some subgraph in $\mathcal M$.  Hence the total running time is in $O(n^{15})$.
\end{proof}
\subsection{Computing $t(9)$ and $c(9)$}\label{subsec:code}
To obtain the values used in the proof of
Theorem~\ref{thm:constant_density}, one can use the following C++
program, which takes a single command line argument, namely the number
$n$ of vertices. We have also checked the computations by hand up to
$n=6$ beyond which they become infeasible. The reader seeking to
reproduce our computer
computation for $n=9$ should be aware that the program requires being
run for several days with typical 2013 hardware.

The efficiency of the program can be significantly improved. However, we decided to keep the code as simple as possible.  Source code for a much more efficient, albeit more complex, version of this program can be obtained from the authors.

\lstset {language=C++,backgroundcolor=\color{white},basicstyle=\scriptsize,breaklines=true,numbers=left, numbersep=5pt,numberstyle=\tiny,showspaces=false, showstringspaces=false,frame=single}

\begin{lstlisting}
#include <vector>
#include <stdio.h>
#include <stdlib.h>
#include <math.h>

using std::vector;

//DECLARATIONS

void Genmatrix(long i);
int IsThick(void);
void Squares(void);
bool Union(void);
bool CheckThick(void);
void Cliques(void);
void Nextvert(vector < int >clique);

//The adjacency matrix:

vector<vector<char> > Adj;

//The vector that will hold thick subgraphs; each graph is a length-n row whose entries are 1 or 0 according to whether the corresponding vertex is in the subgraph:

vector<vector<char> > Thick;

//The number of vertices is n; the number of cliques is clq.

int n;
long clq;
int clqtemp;

main(int argc, char \ast argv[])
{   n = atoi(argv[1]);  //Retrieves the number of vertices from the command line.

    //The following lines declare Adj as an n-by-n matrix.

    Adj.resize(n);
    for (int j = 0; j < n; ++j)
      Adj[j].resize(n);

    long count = 0;

    //For all i at most the number of graphs on a given size-n vertex set, build the adjacency matrix of the i^th graph.  This is accomplished by the function Genmatrix().  The resulting graph is then passed to the function IsThick(), which determines whether it is in the class of thick graphs.  IsThick() returns 1 or 0 according to thickness of the graph, so the variable count is increased by 1 if the graph was thick.  Thus count keeps a count of the number of thick graphs.

    for (long i = 0; i < (long) pow(2.0, n \ast (n - 1) / 2); i++) {
        Genmatrix(i);
        int add = IsThick();
        count += add;

        //If the graph was thick, count how many cliques it contains.  This number is clqtemp, which is added to the running total clq of cliques in thick graphs. We don't keep track of 0- and 1-cliques for the moment.

        if (add == 1) {
            clqtemp = 0;
            Cliques();
            clq += clqtemp;}}

    //Now we add 0- and 1-cliques, i.e. the empty set and the vertices.

    clq=clq+count\ast(n+1);

    //Print the number of thick graphs with a given set of n vertices (i.e. t(n)) and the number of cliques in the disjoint union of all such graphs (i.e. c(n)).

    printf("There are %ld thick graphs with %d vertices\n", count, n);
    printf("There are %ld cliques\n", clq);
}

void Genmatrix(long i)
{    //This function builds the i^th n-by-n symmetric matrix.

     for (int j = 0; j < n; j++) {
        for (int k = 0; k < j; k++) {
            Adj[j][k] = i % 2;
            Adj[k][j] = Adj[j][k];
            i = (long) (i - i % 2) / 2;}}
}

int IsThick()
{   //This function tests a graph for thickness.

    //First, we find all of the induced K_{2,2} subgraphs, and load them into the matrix Thick:

    Squares();

    //If there were no squares, then there are no thick subgraphs, so return 0

    if (Thick.size() == 0)
        return 0;
    else {
        bool u = true;

        //Start taking thick unions and coning off vertices.  Continue to do this (using the function Union) as long as Union is doing things.  Union operates on Thick.

        while (u)
             u = Union();

        //Check if the first line of Thick is all ones, i.e. there is a thick induced subgraph containing all vertices.  If so, return 1.  Otherwise, return 0.

        if (CheckThick())
             return 1;
        else
             return 0;}
}

void Squares()
{   //Clear Thick; we will fill this matrix with squares!  s keeps track of which line of Thick we're in.

    Thick.clear();
    int s = 0;

    //Proceed through all possible pairs of distinct vertices, keeping symmetry in mind.

    for (int i = 0; i < n; i++) {
         for (int j = i + 1; j < n; j++) {

              //We're looking for adjacent i,j; these will form one edge of our square.  Having found such a pair, find a new vertex k that is adjacent to j and not adjacent to i.  Given such a vertex, find a vertex l that completes the square.  Change the current line of Thick to the vector with 1s in the i,j,k,l places and 0s elsewhere.  Move to the next line of Thick and start again.

              if (Adj[i][j] == 1) {
              for (int k = i + 1; k < n; k++) {
                   if (Adj[j][k] == 1 && Adj[i][k] == 0) {
                        for (int l = j + 1; l < n; l++) {
                             if (Adj[i][l] == 1 && Adj[k][l] == 1
                                 && Adj[j][l] == 0) {
                                 s++;
                                 Thick.resize(s);
                                 Thick[s - 1].resize(n);
                                 Thick[s - 1][i] = 1;
                                 Thick[s - 1][j] = 1;
                                 Thick[s - 1][k] = 1;
                                 Thick[s - 1][l] = 1;}}}}}}}
}

bool Union()
{   //This function recognizes new thick subgraphs, given old ones, and modifies Thick accordingly.


    //The variable u is true if we've just performed a non-identity operation on Thick, and false otherwise.  We continue doing operations until u=false.  Again, s is the number of thick subgraphs, i.e. the number of rows in Thick.

    bool u = false;
    int s = Thick.size();

    //Iterate over all pairs of distinct vertices, accounting for symmetry.

    for (int i = 0; i < n; i++) {
         for (int j = i + 1; j < n; j++) {

              //If i,j are non-adjacent, then...

              if (Adj[i][j] == 0) {
                  int k = 0;
                  int first = -1;

                  //...move through the lines in Thick, looking for a thick subgraph containing i and j.  "first" is the identity of the first such subgraph.  If one is found (i.e. first ends up larger than -1), then...

                  while (k < s && first == -1) {
                      if (Thick[k][i] == 1 && Thick[k][j] == 1)
                          first = k;
                      else
                          k++;}

                  //...look among all vertices for p, different from i and j, that is not in the current thick subgraph and is adjacent to i,j.  If found, modify the current row of Thick by adding p; this corresponds to coning off an aspherical subgraph.  We haven't changed the number of rows in Thick, but we've made one bigger.

                  if (first != -1) {
                      for (int p = 0; p < n; p++) {
                          if (p != i && p != j && Thick[first][p] == 0
                              && Adj[i][p] == 1 && Adj[j][p] == 1) {
                              u = true;

                              Thick[first][p] = 1;}}


                      //Remembering i,j, proceed through the rows of Thick, looking for all rows of Thick that contain i and j. Add to the current row any vertex that appears in another row containing i,j, and then remove the row you've just worked on, since its vertices are recorded in the current row Thick[first].

                      while (k < s - 1) {
                         k++;
                         if (Thick[k][i] == 1 && Thick[k][j] == 1) {
                             u = true;
                             for (int p = 0; p < n; p++) {
                                  if (Thick[k][p] == 1)
                                      Thick[first][p] = 1;

                                      Thick[k][p] = Thick[s - 1][p];}
                                      s--;}}}}}}

    Thick.resize(s);
    return u;
}

bool CheckThick()
{   //Return true if and only if the first line of Thick is all 1s.

    int k = 0;
    int j = 0;

    do {
        if (Thick[0][j] == 0)
            k = 1;
            j++;
    } while (k == 0 && j < n);

    if (k == 0)
        return true;
    else
        return false;
}

void Cliques()
{   vector < int >clique;

    //For each j, clear the vector clique, add a new component equal to j, and call Nextvert.  This passes a 1-clique to Nextvert, which will find all cliques containing that clique.

    for (int j = 0; j < n; j++) {
         clique.clear();
         clique.push_back(j);
         Nextvert(clique);}
}

void Nextvert(vector < int >clique)
{   //This function accepts a s-dimensional 0 vector clique, whose entries are the vertices in some clique. The variable j is the last entry in clique.

    int s = clique.size();
    int j = clique[s - 1];

    //For all i between the last entry in clique and the size of the graph, check that the i^th vertex is adjacent to all of the vertices indexed by entries in clique.  If there's a nonadjacency, then adding i won't produce a larger clique, so move to the next i.  Otherwise, put a new entry in clique, equal to i, increment the number of cliques by 1, and pass the new vector to this function.  This terminates at a maximal clique, whereupon we pop up to the previous level of recursion, finish _that_ loop, etc.  In other words, given a clique, this function eventually counts all cliques (with at least two vertices) containing that clique.

    for (int i = j + 1; i < n; i++) {
         bool u = true;

         for (int p = 0; p < s; p++) {
              if (Adj[clique[p]][i] == 0)
                  u = false;}

         if (u) {
            clique.resize(s);
            clique.push_back(i);
            clqtemp++;
            Nextvert(clique);}}
}
\end{lstlisting}

\appendix
\section{Generalizing to all Coxeter groups.\\\small{By J. Behrstock, P.-E. Caprace, M.F. Hagen and A. Sisto}}

All Coxeter groups considered here are assumed finitely generated.
In this section we generalize Theorems~\ref{thmi:rel_hyp_graph} and~\ref{thmi:thick_char} to Coxeter groups which are not necessarily
right-angled.
Further considerations are contained in Subsection~\ref{further}.

We can summarize the main result in this appendix as follows.

\begin{thm}[Minimal relatively
    hyperbolic structures]\label{thm:MinimalPeripheral}
Let $(W, S) $ be a Coxeter system. Then there is a
(possibly empty) collection $\jjj$ of subsets of $S$ enjoying the following properties:
\begin{enumerate}[(i)]
\item The parabolic subgroup $W_J$ is strongly algebraically thick  for every $J \in \jjj$.

\item If $J\neq S$ for all $J\in\jjj$, then $W$ is hyperbolic relative to
$\ppp = \{W_J \; | \; J \in \jjj\}$.

\end{enumerate}

In particular $\ppp$ is a minimal relatively hyperbolic structure for $W$.
\end{thm}

\subsection{Thick Coxeter groups}\label{subsec:appthickcoxeter}
We consider the class $\coxeterthick$ of Coxeter systems $(W,S)$ defined as follows.

\begin{enumerate}
 \item $\coxeterthick$ contains the class $\coxeterthick_0$ of all irreducible \emph{affine} Coxeter systems $(W,S)$ with $S$ of cardinality~$\geq 3$, as well as all Coxeter systems of the form $(W,S_1 \cup S_2)$ with  $W_{S_1}, W_{S_2}$ irreducible non-spherical and $[W_{S_1}, W_{S_2}]=1$.
 \item Suppose that $(W, S \cup {s})$ is such that  ${s}^\perp$ is non-spherical and $(W_S, S)$ belongs to $\coxeterthick$. Then $(W, S \cup {s})$ belongs to $\coxeterthick$.
 \item Suppose that $(W,S)$ has the property that there exist $S_1,S_2\subseteq S$ with $S_1\cup S_2=S$, $(W_{S_1},S_1), (W_{S_2},S_2)\in \coxeterthick$ and  $W_{S_1\cap S_2}$ non-spherical. Then $(W,S)\in \coxeterthick$.
\end{enumerate}


\begin{prop}\label{prop:thickgeneral}
For $(W,S)\in \coxeterthick$, the Coxeter group $W$ is strongly algebraically thick.
\end{prop}

The proof requires the following subsidiary fact.

\begin{lem}\label{lem:index2}
Let $(W, S)$ be a Coxeter system. Let $s \in S$ and set $K = S \setminus \{s\}$. Then the group $\la W_K \cup s W_K s\ra$ has index at most~$2$ in $W$.
\end{lem}

\begin{proof}
The group $\la W_K \cup s W_K s\ra$ is a reflection subgroup whose fundamental domain for its action on the Cayley graph of $(W, S)$ contains at most two chambers, namely the base vertex $1$ and the unique vertex $s$-adjacent to it, see \cite{Deod_refl}.
\end{proof}

\begin{proof}[Proof of Proposition \ref{prop:thickgeneral}]
 If $(W,S)$ is in $\coxeterthick_0$ then the group $W$ is either
 virtually abelian of rank~$\geq 2$, or a direct product of two
 infinite (Coxeter) groups.  In particular $W$ is wide and, hence, strongly
 algebraically thick of order $0$.

 Let $(W,S\cup\{s\})$ be of the form described in item $2)$ of the definition of $\coxeterthick$. Lemma~\ref{lem:index2} then implies that $W$ contains the group $\la W_S \cup sW_Ss \ra$ with index at most~$2$. Therefore $W$ is strongly algebraically thick, being an algebraic network with respect to the pair of strongly thick groups $\{W_S, sW_Ss\}$.

 Finally, let $(W,S)$ be as in item $3)$ of the definition of $\coxeterthick$. Then $W$ is is strongly algebraically thick, being an algebraic network with respect to the pair of strongly thick groups $\{W_{S_1}, W_{S_2}\}$.
\end{proof}

\subsection{Proof of minimal relatively hyperbolic structures theorem}

We will use the following criterion for relative hyperbolicity of
Coxeter groups, which corrects
\cite[Theorem~A]{Caprace:relatively_hyperbolic} where a
hypothesis on the peripheral subgroups was missing.

\begin{thm}\cite[Theorem~A$'$]{Caprace:relatively_hyperbolicErr}
\label{a'}
 Let $(W,S)$ be a Coxeter system and $\jjj$ a collection of proper subsets of $S$. Then $W$ is hyperbolic relative to $\{W_J|J\in \jjj\}$ if and only if the following conditions hold:

 \medskip (RH1) For each irreducible affine subset $K\subseteq S$ of
 cardinality at least $3$, there exists $J\in \jjj$ so that
 $K\subseteq J$.  Similarly, given any pair of irreducible
 non-spherical subsets $K_1,K_2\subseteq S$ with $[K_1,K_2]=1$, there
 exists $J\in \jjj$ so that $K_1\cup K_2\subseteq J$.

 \medskip (RH2) For all $J_1,J_2\in\jjj$ with $J_1\neq J_2$, the
 intersection $J_1\cap J_2$ is spherical.

 \medskip (RH3) For each $J\in\jjj$ and each irreducible non-spherical
 $K\subseteq J$, we have $K^\perp\subseteq J$.
\end{thm}

We are now ready to prove Theorem~\ref{thm:MinimalPeripheral}. We will actually give an explicit description of $\jjj$:

\begin{thm}\label{thm:MinimalPeripheral+}
Let $(W, S) $ be a Coxeter system and let $\jjj$ be the (possibly empty) collection of all \emph{maximal} subsets $J\subseteq S$ so that $(W_J,J)\in \coxeterthick$. Then:
\begin{enumerate}[(i)]
\item The parabolic subgroup $W_J$ is strongly algebraically thick  for every $J \in \jjj$.

\item If $\jjj\neq \{S\}$, then $W$ is hyperbolic relative to
$\ppp = \{W_J \; | \; J \in \jjj\}$.
\end{enumerate}

In particular $\ppp$ is a minimal relatively hyperbolic structure for $W$.
\end{thm}

\begin{proof}
By Moussong's characterization of hyperbolic Coxeter groups~\cite[Theorem~17.1]{Moussong:thesis} (and the fact that $S$ is finite), $\jjj$ is not empty if and only if $W$ is not hyperbolic, which we assume from now on.

By Proposition~\ref{prop:thickgeneral}, (i) holds.

We are now left to show that $\jjj$ satisfies the three conditions (RH1)--(RH3) from Theorem~\ref{a'}.

It is clear that $\jjj$ satisfies (RH1).

If $J_1,J_2\in\jjj$ are distinct then $W_{J_1\cap J_2}$ must be spherical. In fact, if it was non-spherical then we would have $J_1\cup J_2\in \jjj$, contradicting the maximality of either $J_1$ or $J_2$. So, $\jjj$ satisfies (RH2).

Let $K$ be a non-spherical subgraph of some $J\in \jjj$. We have to show that $K^\perp$ is contained in $J$ as well. Indeed, if there was an element $s\in K^\perp\backslash J$, then $J\cup \{s\}$ would be in $\coxeterthick$, contradicting the maximality of $J$.

We have now shown the peripherals are in $\mathbb T$ and hence thick by
Proposition~\ref{prop:thickgeneral}. Thus, as noted in the introduction,
minimality now follows from \cite[Corollary~4.7]{BDM}.
\end{proof}

\subsection{Intrinsic horosphericity and further corollaries}\label{further}

We say that a discrete group $\Gamma$ is \emph{(intrinsically) horospherical} if every proper isometric action of $\Gamma$ on a proper hyperbolic geodesic metric space fixes a unique point at infinity. In particular the group $\Gamma$ cannot be virtually cyclic, and every element of infinite order acts as a parabolic isometry in any such  $\Gamma$-action. As one may expect, thickness and horosphericity are related properties (compare Theorem~4.1 from \cite{BDM}):

\begin{prop}\label{prop:Thick=>Periph}
Every strongly algebraically thick group is intrinsically horospherical.
\end{prop}

The proof requires the following result, which follows from the exact same arguments as the proof of Lemma~3.25 in \cite{DrutuMozesSapir}.

\begin{lem}\label{lem:DMS}
Let $H$ be a finitely generated group (endowed with its word metric with respect to a finite generating set), $(X, d)$ be a metric space and $\mathfrak q \colon H \to X$ be a map which is Lipschitz up to an additive constant. Given $h \in H$, if the map $\mathbf Z \to X : n \mapsto \mathfrak q(h^n)$ is a Morse quasi-geodesic in
$X$, then $h$ is a Morse element in $H$. \qed
\end{lem}

\begin{lem}\label{lem:UniqueFP}
Let $H$ be a group acting properly by  isometries on a proper Gromov hyperbolic metric space $X$. Assume that $H$ has a unique fixed point $\xi$ at infinity of $X$. Then every infinite subgroup of $H$ has $\xi$ as its unique fixed point at infinity.
\end{lem}

\begin{proof}
The hypotheses imply that $H$ does not contain any hyperbolic isometry. From Proposition~5.5 in \cite{CF}, it follows that every subgroup of $H$ either has a bounded orbit, or has a unique fixed point at infinity of $X$. The desired conclusion follows since the $H$-action on $X$ is proper.
\end{proof}

\begin{proof}[Proof of Proposition~\ref{prop:Thick=>Periph}]
Let $H$ be a finitely generated group which is wide.
Suppose that $H$ acts properly by isometries on a proper Gromov
hyperbolic metric space $X$.
$H$ can not contain a hyperbolic isometry,  since otherwise
Lemma~\ref{lem:DMS} implies that some
asymptotic cone of $H$ has cut-points, which would contradict the
assumption that $H$ is wide.  Since $H$ is infinite and the
$H$--action on $X$ is proper, it follows from
\cite[Proposition~5.5]{CF} that $H$ fixes a unique point at infinity
of $X$.  This proves that strongly algebraically thick groups of
order~$0$ are intrinsically horospherical.

The desired conclusion now follows by induction on the order of
thickness, the induction step being given by the following
observation.  Let $G$ be an infinite group which is an $M$-algebraic
network with respect to a finite collection $\mathcal H$ of subgroups.
If each subgroup in $\mathcal H$ is intrinsically
peripheral, then so is $G$.

Indeed, let $G$ act properly by isometries on  a proper Gromov hyperbolic metric space $X$. Then each group $H \in \mathcal H$ has a unique fixed point $\xi_H$ at infinity of $X$. Given $H, H' \in \mathcal H$, there is a sequence $H = H_1, \dots, H_N = H'$ in $\mathcal H$ in which any two consecutive groups have an infinite intersection, see Definition~5.2 in \cite{BDM}. From Lemma~\ref{lem:UniqueFP}, we deduce that $\xi_{H} = \xi_{H_1} = \dots = \xi_{H_n} = \xi_{H'}$. Hence all groups in $\mathcal H$ have the same fixed point at infinity, say $\xi$. By the definition of an algebraic network, this point $\xi$ must be fixed by a finite index subgroup of $G$. Thus the $G$-orbit of $\xi$ is finite. But if that orbit contains more than two points then $G$ will have a bounded orbit, contradicting the fact that $G$ is infinite and acts properly. Similarly, if the orbit contains exactly two points, then $G$ is virtually cyclic and hence does not contain any intrinsically peripheral subgroup, which is absurd. Thus $G$ fixes $\xi$ (and no other point at infinity of $X$).
\end{proof}

Notice that the converse to Proposition~\ref{prop:Thick=>Periph} does
not hold in general: indeed horospherical groups include all amenable
groups that are not virtually cyclic. In particular, infinite locally
finite groups are examples of horospherical groups that are not
strongly algebraically thick.
By Zorn's lemma, every intrinsically horospherical subgroup of
$\Gamma$ is contained in a maximal one.  It is thus a natural question
to determine all the maximal intrinsically horospherical subgroups.
Theorem~\ref{thm:MinimalPeripheral} yields the answer to
this question when $\Gamma$ is a Coxeter group.

\begin{cor}\label{cor:MinimalPeripheral}
Let $W$ be a Coxeter group. Then the maximal intrinsically horospherical subgroups of $W$ are parabolic subgroups (in the sense of Coxeter group theory) with respect to any Coxeter generating set. Those parabolic subgroups are precisely the conjugates of the elements of the set $\ppp$  afforded by Theorem~\ref{thm:MinimalPeripheral}.
\end{cor}

\begin{proof}
Every strongly algebraically thick group is intrinsically horospherical by Proposition~\ref{prop:Thick=>Periph}. Moreover, a subgroup of $W$ containing properly a conjugate of an element of $\ppp$  cannot be intrinsically horospherical by Theorem~\ref{thm:MinimalPeripheral}. Thus the elements of $\ppp$ are indeed maximal horospherical subgroups. Since $W$ is relatively hyperbolic with respect to $\ppp$, every intrinsically horospherical subgroup is conjugate to a subgroup of an element of $\ppp$.
\end{proof}

\begin{cor}\label{cor:charThick}
Let $(W, S)$ be a Coxeter system. Then the following conditions are equivalent:
\begin{enumerate}[(i)]
\item $(W,S)$ is in $\coxeterthick$\label{it:inT}

\item $W$ is strongly algebraically thick; \label{it:thick}

\item $W$ is intrinsically horospherical; \label{it:inperiph}

\item $W$ is not relatively hyperbolic with respect to any family of
proper subgroups; \label{it:nonRH}

\item $W$ is not relatively hyperbolic with respect to any family of
proper Coxeter-parabolic subgroups; \label{it:coxnonRH}

\item For every collection $\jjj$ of subsets of $S$ satisfying
(RH1)--(RH3), we have $S \in \jjj$. \label{it:parab}
\end{enumerate}

\end{cor}

\begin{proof}
The implication (\ref{it:inT}) $\Rightarrow$ (\ref{it:thick}) is the
content of Proposition \ref{prop:thickgeneral}.  The implication
(\ref{it:thick}) $\Rightarrow$ (\ref{it:inperiph}) follows from
Proposition~\ref{prop:Thick=>Periph}.  The implication
(\ref{it:inperiph})~$\Rightarrow$~(\ref{it:nonRH}) is straightforward.
Property (\ref{it:nonRH}) trivially implies (\ref{it:coxnonRH}).
That (\ref{it:coxnonRH}) is equivalent to (\ref{it:parab}) follows from
Theorem~\ref{a'}.
Applying
Theorem~\ref{thm:MinimalPeripheral+} we get that
(\ref{it:coxnonRH}) implies (\ref{it:inT}).
\end{proof}


\bibliographystyle{alpha}
\bibliography{Coxeter_Thick.bib}
\end{document}